\documentclass[11pt]{amsart}
\usepackage{amsmath}
\usepackage{latexsym,amsfonts,amssymb,mathrsfs}
\usepackage{mathdots}
\usepackage{color}
\usepackage[all,2cell,color]{xy}
\usepackage{enumitem} 
\usepackage{bbm}
\usepackage{tikz}
\usepackage{tikz-cd}
\usepackage[T1]{fontenc}
 \usepackage[utf8]{inputenc}
 \usepackage{blkarray} 
 \usepackage{mathtools} 
 \usepackage{longtable}
 \usepackage{multicol} 
 \usepackage{hyperref}
\usepackage[capitalise]{cleveref}

\usetikzlibrary{positioning} 
\usetikzlibrary{decorations.pathreplacing}
\usetikzlibrary{arrows, decorations.markings} 
\pgfarrowsdeclarecombine{twotriang}{twotriang}{stealth}{stealth}%
{stealth}{stealth}
\tikzset{arrow/.style={-stealth}}
\tikzset{arrowshorter/.style={-stealth, shorten <=2pt, shorten >=2pt}}
\tikzset{arrowmuchshorter/.style={-stealth, shorten <=7pt, shorten >=6pt}}
\tikzset{mono/.style={>-stealth}} 
\tikzset{epi/.style={-twotriang}} 
\tikzset{twoarrowlonger/.style={double,double distance=1.5pt,
shorten <=5pt,shorten >=6pt,
decoration={markings,mark=at position -4pt with {\arrow[scale=1.75]{>}}},
preaction={decorate}}} 

\tikzset{twoarrow/.style={double,double distance=1.5pt,
shorten <=6pt,shorten >=7pt, 
decoration={markings,mark=at position -4pt
with {\arrow[scale=1.75]{>}}},
preaction={decorate} 
}
}
\tikzset{%
    symbol/.style={%
        draw=none,
        every to/.append style={%
            edge node={node [sloped, allow upside down, auto=false]{$#1$}}}
    }
}

\tikzset{mapstikz/.style={-stealth, 
decoration={markings,mark=at position 0pt with {\arrow[scale=0.5]{|}}}, preaction={decorate}}}
\usetikzlibrary{backgrounds,shapes,arrows,calc,patterns} 

\pgfdeclarelayer{background}
\pgfsetlayers{background,main}

\theoremstyle{plain}   
\newtheorem{thm}{Theorem}[section] 
\makeatletter\let\c@thm\c@thm\makeatother
\newtheorem{cor}{Corollary}[section]
\makeatletter\let\c@cor\c@thm\makeatother
\newtheorem{lem}{Lemma}[section]
\makeatletter\let\c@lem\c@thm\makeatother
\newtheorem{prop}{Proposition}[section]
\makeatletter\let\c@prop\c@thm\makeatother

\makeatletter\let\c@claim\c@thm\makeatother

\makeatletter\let\c@conjecture\c@thm\makeatother

\newtheorem*{unnumberedproposition}{Proposition}

\newtheorem{thmalph}{Theorem}

\theoremstyle{definition}

\newtheorem{defn}{Definition}[section]
\makeatletter\let\c@defn\c@thm\makeatother

\makeatletter\let\c@const\c@thm\makeatother

\makeatletter\let\c@notn\c@thm\makeatother
\newtheorem{convention}{Convention}[section]
\makeatletter\let\c@convention\c@thm\makeatother

\theoremstyle{remark}

\newtheorem{rmk}{Remark}[section]
\makeatletter\let\c@rmk\c@thm\makeatother
\newtheorem{ex}{Example}[section]
\makeatletter\let\c@ex\c@thm\makeatother

\makeatletter\let\c@observation\c@thm\makeatother

\makeatletter\let\c@warning\c@thm\makeatother

\makeatletter\let\c@digression\c@thm\makeatother

\makeatletter\let\c@answ\c@thm\makeatother

\makeatletter\let\c@answ\c@thm\makeatother
\newtheorem{aside}{Aside}[section]
\makeatletter\let\c@aside\c@thm\makeatother

\makeatletter
\let\c@equation\c@thm
\numberwithin{equation}{section}
\makeatother

\newcommand{\newrefformat}[2]{}

\crefname{lem}{Lemma}{Lemmas}
\crefname{thm}{Theorem}{Theorems}
\crefname{defn}{Definition}{Definitions}
\crefname{notn}{Notation}{Notations}
\crefname{const}{Construction}{Constructions}
\crefname{prop}{Proposition}{Propositions}
\crefname{rmk}{Remark}{Remarks}
\crefname{cor}{Corollary}{Corollaries}
\crefname{equation}{Display}{Displays}
\crefname{ex}{Example}{Examples}
\crefname{thmalph}{Theorem}{Theorems}
\crefname{answ}{Answer}{Answers}
\crefname{crzcond}{Crazy Condition}{Crazy Conditions}

\newcommand{\cC}{\mathcal{C}}
\newcommand{\cD}{\mathcal{D}}

\newcommand{\cP}{\mathcal{P}}

\newcommand{\cS}{\mathcal{S}}
\newcommand{\cT}{\mathcal{T}}

\newcommand{\cat}{\cC\!\mathit{at}}
\newcommand{\set}{\cS\!\mathit{et}}

\DeclareMathOperator{\Map}{Map}

\newcommand{\inttrunc}[1]{\tau_{\leq n}^{\text{i}}}

   \newcommand{\matn}[2]{\operatorname{Mat}_{#1}{#2}}
      \newcommand{\mat}[1]{\operatorname{Mat}{#1}}

\DeclareMathOperator{\id}{id}

\DeclareMathOperator{\Ob}{Ob}
\DeclareMathOperator{\op}{op}

\DeclareMathOperator{\sk}{sk}

\tikzset{%
scalearrow/.style n args={3}{
  decoration={
    markings,
    mark=at position (1-#1)/2*\pgfdecoratedpathlength
      with {\coordinate (#2);},
    mark=at position (1+#1)/2*\pgfdecoratedpathlength
      with {\coordinate (#3);},
    },
  postaction=decorate,
  }
}

\newcommand\simpfverArrow[9]{%
    \def\tempa{#1}%
    \def\tempb{#2}%
    \def\tempc{#3}%
    \def\tempd{#4}%
    \def\tempe{#5}%
    \def\tempf{#6}%
    \def\tempg{#7}%
    \def\temph{#8}%
    \def\tempi{#9}%
}
\newcommand\simpfverArrowcontinued[5]{%
    \begin{tikzcd}[column sep=1.3cm, row sep=0.4cm, baseline=(current  bounding  box.center), ampersand replacement=\&]
 #5 \& #4 \arrow[l, "\tempc" swap] \&[-7mm] \&[-7mm] #5 \& #4 \arrow[l, "\tempc" swap]\\
 {} \& \&[-7mm] {=} \&[-7mm] \& \\
 #2 \arrow[r, "\tempa" swap]\arrow[uu, "\tempd", ""{name=a031,inner sep=2pt, swap} ] \arrow[uur, "\tempe"{description}, ""{name=a02,inner sep=2pt, swap}]\& #3 \arrow[uu, "\tempb" swap]\&[-7mm] \&[-7mm] #2 \arrow[uu, "\tempd", ""{name=a032,inner sep=2pt, swap}]\arrow[r, "\tempa" swap]\& #3\arrow[uu, "\tempb" swap] \arrow[uul, "\tempf"{description}, ""{name=a13,inner sep=2pt, swap}]
 \arrow[phantom, from=a031, to=1-2, scalearrow={0.63*0.8}{start023}{end023}]
 \arrow[Rightarrow, to path={(start023)--(end023)\tikztonodes}, "\tempg" ]
  \arrow[phantom, from=a02, to=3-2, scalearrow={0.8}{start012}{end012}]
 \arrow[Rightarrow, to path={(start012)--(end012)\tikztonodes}, "\temph" swap ]
  \arrow[phantom, from=a032, to=3-5, scalearrow={0.63*0.8}{start013}{end013}]
 \arrow[Rightarrow, to path={(start013)--(end013)\tikztonodes}, "\tempi" swap ]
   \arrow[phantom, from=a13, to=1-5, scalearrow={0.8}{start123}{end123}]
 \arrow[Rightarrow, to path={(start123)--(end123)\tikztonodes}, "#1" ] 
 \end{tikzcd}
}

\tikzset{mono/.style={>-stealth}} 
\tikzset{epi/.style={-twotriang}}

\author{Viktoriya Ozornova}
\address{Fakult\"at f\"ur Mathematik, Ruhr-Universit\"at Bochum, Bochum, Germany}
\email{viktoriya.ozornova@rub.de}

\author{Martina Rovelli}
\address{Mathematical Sciences Institute,
The Australian National University,
Canberra, Australia
}
\email{martina.rovelli@anu.edu.au}

\begin{document}

\title{The Duskin nerve of $2$-categories in Joyal's cell category $\Theta_2$}

\maketitle

\begin{abstract}
We give an explicit and purely combinatorial description of the Duskin nerve of any $(r+1)$-point suspension $2$-category, and in particular of any $2$-category belonging to Joyal's cell category $\Theta_2$.
\end{abstract}

\tableofcontents

\section*{Overview of results}

A $2$-categorical analog of the nerve of ordinary categories goes by the name of \emph{Duskin nerve}. It first occurred as an instance of Street's nerve of $\omega$-categories from \cite{StreetOrientedSimplexes}, and was then studied in detail by Duskin in \cite{duskin}.
Roughly speaking, the objects, $1$-morphisms and $2$-morphisms of the given $2$-category are incorporated suitably in the $0$-, $1$- and $2$-simplices of the Duskin nerve, which is $3$-coskeletal.

The Duskin nerve is a classical construction, and many of its homotopical properties have been established.
For instance, Duskin \cite{duskin} showed that the Duskin nerve of a $(2,0)$-category is always a Kan complex and that the Duskin nerve of a $(2,1)$-category is always a quasi-category. Bullejos, Carrasco, Cegarra, and Garz\'{o}n showed in different combinations that analogs of Quillen's Theorems A and B hold for the Duskin nerve of $2$-categories \cite{BullejosCegarra,cegarra}, and that the Duskin nerve is homotopically equivalent to other nerve constructions for $2$-categories \cite{CCG}. To mention one application, Nanda \cite{nanda} then built on their work showing that the Duskin nerve of the discrete flow $2$-category associated to a simplicial complex (with extra structure) has the same homotopy type as that of the simplicial complex.

While the machinery developed by Steiner \cite{ SteinerOrientals, SteinerSimpleOmega} implicitly provides general methods to study the Duskin nerve of $2$-categories, the combinatorics behind this nerve still remains quite mysterious. In \cite{BuckleyGarnerLackStreet}, Buckley, Garner, Lack and Street show that the Duskin nerve of a rather simple $2$-category (the monoidal category $([1], \min, 1)$) is the highly non-trivial ``Catalan simplicial set''.
Even for very simple $2$-categories we are however not aware of explicit computations and descriptions of the Duskin nerve.

As a first analysis in this direction, one could observe that in the nerve of finite $1$-categories with no non-identity endomorphisms there are only finitely many non-degenerate simplices, and imagine that the Duskin nerve of finite $2$-categories would enjoy the same property.
Somewhat surprisingly, we discovered that the Duskin nerve of $2$-categories is much more complex than expected. For instance, we show in \cref{section2} that the Duskin nerve of 
the free $2$-cell
\[
\begin{tikzcd}[row sep=3cm, column sep=2cm]
  x \arrow[r, bend left, "f_0", ""{name=U,inner sep=2pt,below}]
  \arrow[r, bend right, "f_1"{below}, ""{name=D,inner sep=2pt}]
  & y,
  \arrow[Rightarrow, from=U, to=D, "\alpha"]
\end{tikzcd}
\]
which is
a very simple $2$-category that does not contain any non-trivial composition,
has non-degenerate simplices in each dimension.

\begin{unnumberedproposition}
 \label{intrinsecdescription}
The Duskin nerve of the free $2$-cell has precisely two non-degenerate simplices in each positive dimension.
\end{unnumberedproposition}

This result was unexpected to us, and we were able to conjecture it in the first place only after having a computer produce all $n$-simplices of the Duskin nerve of the free $2$-cell for $n\le 6$.
In order to prove the proposition, we developed a more general study of the Duskin nerve of $2$-categories of the form $\Sigma\cD$, sometimes referred to as \emph{suspension $2$-categories}, of which the free $2$-cell is an example for $\cD=[1]$. 

The suspension $2$-category $\Sigma\cD$ of a category $\cD$, which can be pictured as
\[\Sigma\cD\ :=\ \ \ \ 
    \begin{tikzpicture}[baseline=-5pt]
    \draw (0,0)  node[inner sep=0.2cm](a){} node(x){$x$};
    \draw (2,0) node[inner sep=0.2cm](b){} node(y){$y$};
     \draw[->] (a) edge[bend right] node[below](xy){$\cD$}(b);
     \draw[->] (b) edge[bend right] node[above](xy){$\varnothing$}(a);
     \draw[->] (a.180) arc (30:330:2mm)node[xshift=-0.8cm](xx){$[0]$};
     \draw[->] (b.1) arc (-210:-510:2mm)node[xshift=0.8cm](yy){$[0],$};
    \end{tikzpicture}
\]
appears often in the literature as a special case of a simplicial suspension. For instance,
the homwise nerve $N_*(\Sigma\cD)$ of the suspension $\Sigma\cD$ is a simplicial category that agrees with what would be denoted as $U(N\cD)$ in \cite{bergner}, as $S(N\cD)$ in \cite{Joyal2007}, as $[1]_{N\cD}$ in \cite{htt}, 
and as ${\mathbbm{2}}[N\cD]$ in \cite{RiehlVerityNcoh}.

In \cref{section1} we
prove as \cref{simplicesasmatrices} the following description for the Duskin nerve of suspension $2$-categories.

\begin{thmalph}
\label{ThmA}
Let $\cD$ be a $1$-category.
An $n$-simplex of the Duskin nerve of the suspension $\Sigma{\cD}$ can be uniquely described as a functor $\sigma\colon[k]\times[l]^{\op}\to\cD$ together with $k,l\ge-1$ and $k+l=n-1$, which can be pictured as a ``matrix'' valued in $\cD$ of the form
\[{
\small \begin{tikzcd}
 d_{0l} \arrow{d} \arrow{r} & d_{0 (l-1)} \arrow{d} \arrow{r} &  \cdots \arrow{r}   & d_{0 0} \arrow{d}\\
  d_{1 l} \arrow{d}\arrow{r} & d_{1 (l-1)} \arrow{d} \arrow{r} &  \cdots \arrow{r}  & d_{1 0 } \arrow{d}\\
\vdots \arrow{d}& \vdots \arrow{d}  & \ddots & \vdots \arrow{d} \\
d_{k l} \arrow{r}  & d_{k (l-1)} \arrow{r}  & \cdots   \arrow{r} & d_{k 0},
\end{tikzcd}
}\]
and the simplicial structure is understood as suitably removing or doubling rows or columns.
\end{thmalph}

After having understood the Duskin nerve of suspension $2$-categories, we then study the Duskin nerve of $(r+1)$-point suspension $2$-categories $\Sigma[\cD_1,\dots,\cD_r]$, which are $2$-categories obtained by pasting together suspension $2$-categories of categories $\cD_1\dots,\cD_r$ along objects as in the following picture:
\[\Sigma[\cD_1,\dots,\cD_{r}]\ :=
    \begin{tikzpicture}[baseline=0pt]
    \draw (0,0)  node[inner sep=0.2cm](a){} node(x0){$x_0$};
    \draw (2,0) node[inner sep=0.2cm](b){} node(x1){$x_1$};
    \draw (4,0) node[inner sep=0.2cm](c){} node(x2){$x_2$};
    \draw (6,0) node[inner sep=0.2cm](d){} node(dots){$\ldots$};
    \draw (8,0) node[inner sep=0.2cm](e){} node(x_r){$x_r$.};
     \draw[->] (a) to[bend right] node[below](xy1){$\cD_1$}(b);
     \draw[->] (b) to[bend right] node[above](xy2){$\varnothing$}(a);
       \draw[->] (b) to[bend right] node[below](xy3){$\cD_2$}(c);
     \draw[->] (c) to[bend right] node[above](xy4){$\varnothing$}(b);
       \draw[->] (c) to[bend right] node[below](xy5){$\cD_3$}($(d)+(-0.4, -0.1)$);
     \draw[->] ($(d)+(-0.4, 0.1)$) to[bend right] node[above](xy6){$\varnothing$}(c);
       \draw[->] ($(d)+(0.4, -0.1)$) to[bend right] node[below](xy7){$\cD_r$}(e);
     \draw[->] (e) to[bend right] node[above](xy8){$\varnothing$}($(d)+(0.4, 0.1)$);
    \end{tikzpicture}
\]

This type of horizontal gluing of suspension $2$-categories appears e.g.~in \cite[\textsection 10.5]{VerityComplicialAMS}, and is an instance of a $2$-category freely generated by a $\cat$-graph, a construction which goes all the way back to \cite{wolff}.
Motivating examples of $(r+1)$-point suspension $2$-categories are the $2$-categories that belong to Joyal's cell category $\Theta_2$ from \cite{JoyalDisks}, which are all $(r+1)$-point suspension $2$-categories of the form $[r|n_1,\dots,n_r]\cong\Sigma[[n_1],\dots,[n_r]]$
for $n_1,\dots,n_r\ge0$. An example would be the $2$-category $[3|2,0,1]$, which is generated by the following data:
\[
\begin{tikzcd}[row sep=3.2cm, column sep=2.2cm]
  x \arrow[r, bend left=50, "f", ""{name=U,inner sep=2pt,below}]
  \arrow[r, "g"{near end, xshift=0.2cm}, ""{name=D,inner sep=2pt},""{name=M,inner sep=2pt, below}]
  \arrow[r, bend right=50, "h"{below}, ""{name=DD,inner sep=2pt, xshift=0.05cm}]
  & y\arrow[r, "l"]
    & [-0.9cm]
    z\arrow[r, bend left=50, "m", ""{name=U1,inner sep=2pt,below}]
  \arrow[r, "k"{near end, xshift=0.2cm}, ""{name=D1,inner sep=2pt, xshift=0.05cm},""{name=M1,inner sep=2pt, below}]
  & w.
    \arrow[Rightarrow, from=U, to=D, "\alpha"]
    \arrow[Rightarrow, from=M, to=DD, "\beta"{near start}]  
  \arrow[Rightarrow, from=U1, to=D1, "\gamma"]
\end{tikzcd}\]

We are able to describe the Duskin nerve of $(r+1)$-point suspension $2$-categories in \cref{section3}; the precise statement will appear as \cref{nerver+1suspension}.

\begin{thmalph}
\label{ThmB}
Let $\cD_1,\dots,\cD_r$ be given $1$-categories.
An $n$-simplex of the Duskin nerve of the $(r+1)$-point suspension $\Sigma[\cD_1,\dots,\cD_r]$ can be uniquely described as an $(r+1)$-uple of
$\cD_i$-valued matrices whose numbers of rows are suitably increasing.
\end{thmalph}

 The explicit description of the Duskin nerve of $2$-categories from this paper can then also be used to prove finer homotopical properties. 
For instance, in ongoing work \cite{ORfundamentalpushouts}, we use these results to show that the canonical inclusion
\[N(\Sigma[1])\underset{N(\Sigma[0])}{\amalg}N(\Sigma[1]) \hookrightarrow N(\Sigma[2])\]
is a categorical equivalence, and even a weak equivalence of $(\infty,2)$-categories in the model of $2$-complicial sets.

\addtocontents{toc}{\protect\setcounter{tocdepth}{1}}
\subsection*{Acknowledgements}
We would like to thank Clark Barwick, Andrea Gagna, Lennart Meier and Emily Riehl for helpful conversations, and the referee for insightful comments.

  \section{The Duskin nerve of suspension $2$-categories}

\label{section1}
\label{descriptionDuskinnerve}

We start by recalling the definition of the Duskin nerve of $2$-categories\footnote{In this paper we are only concerned with \emph{strict} $2$-categories.}.

\begin{defn}
The \emph{Duskin nerve} $N(\cC)$ of a $2$-category $\cC$ is a $3$-coskeletal simplicial set in which
\begin{enumerate}[leftmargin=*]
\item[(0)] a $0$-simplex consists of an object of $\cC$:
$$x;$$
    \item a $1$-simplex consists of a $1$-morphism of $\cC$:
     \[
\begin{tikzcd}
    x \arrow[rr, "a"{below}]&& y;
\end{tikzcd}
\]
    \item a $2$-simplex consists of a $2$-cell of $\cC$ of the form $c\Rightarrow b\circ a$:
    $$\begin{tikzcd}[baseline=(current  bounding  box.center)]
 & y \arrow[rd, "b"]  & \\
    x \arrow[ru, "{a}"]
  \arrow[rr, "c"{below}, ""{name=D,inner sep=1pt}]
  && z;
  \arrow[Rightarrow, from=D, 
 to=1-2, shorten >= 0.1cm, shorten <= 0.1cm, ""]
\end{tikzcd}$$
\item a $3$-simplex consists of four $2$-cells of $\cC$ that satisfy the following relation.
\[
\simpfverArrow{d}{c}{e}{a}{b}{f}{}{}{}\simpfverArrowcontinued{}{x}{y}{z}{w}
\]
\end{enumerate}
The simplicial structure of $N\cC$ is as indicated in the pictures.
\end{defn}

The following type of $2$-category is of interest in this paper. We denote by $[-1]$ the empty category.

\begin{defn}
\label{suspension}
The \emph{suspension} of a $1$-category $\cD$ is the $2$-category $\Sigma\cD$ with two objects $x,y$ and hom categories given by
\[
\Map_{\Sigma\cD}(x,y)=\cD,\ \Map_{\Sigma\cD}(y,x)=[-1],\ \Map_{\Sigma\cD}(x,x)=\Map_{\Sigma\cD}(y,y)=[0].
\]
\end{defn}

\begin{rmk}
Given the fact that $\Sigma\cD$ has only two objects, for $n\ge2$ any $n$-simplex $\sigma$ in $N\Sigma\cD$ has exactly zero or one non-degenerate $1$-simplices of the form $(k,k+1)$. More precisely,
\begin{itemize}[leftmargin=*]
    \item[(a)] if the simplex $\sigma$ is the degeneracy of one of the $0$-simplices $x$ or $y$, each edge of $\sigma$ is degenerate at the same vertex $x$ or $y$.
    \item[(b)] if the simplex is not the degeneracy of a $0$-simplex, it has precisely one non-degenerate edge of the form $(k,k+1)$ for some $0\le k\le n-1$.
\end{itemize}
\end{rmk}

The fact that these two very different behaviours partition the simplices of $N\Sigma\cD$ is fundamental, and we therefore make the following definition.

\begin{defn}
For $n\ge1$ we say that an $n$-simplex $\sigma$ in $N\Sigma\cD$ is
\begin{itemize}[leftmargin=*]
    \item[(a)] \emph{maximally degenerate} if it is the degeneracy of one of the $0$-simplices $x$ or $y$.
    \item[(b)] \emph{of type $k$} for $0\le k\le n-1$ if it has one non-degenerate $1$-simplex of the form $(k,k+1)$.
\end{itemize}
\end{defn}

In particular, it is consistent to think of the maximally degenerate $n$-simplex of the $0$-simplex $y$ as the (unique) $n$-simplex \emph{of type $-1$} and the maximally degenerate $n$-simplex of the $0$-simplex $x$ as the (unique) $n$-simplex \emph{of type $n$}, with the type $k$ being determined by the formula
\[k=\left\{
    \begin{array}{ll}
    -1&\text{if $\sigma$ is constant at $y$}\\
       \max\{0\le i\le n\ |\ \sigma(i)=x\}  &  \text{if $\sigma$ is not constant}\\
      n   & \text{if $\sigma$ is constant at $x$.}
    \end{array}
    \right.\]

\begin{ex}
The $3$-simplex of $N\Sigma\cD$
\[
\simpfverArrow{\id_x}{c}{\id_y}{a}{b}{f}{\varphi}{\gamma}{\psi}\simpfverArrowcontinued{\theta}{x}{x}{y}{y},
\]
is of type $1$, whereas the $3$-simplex
\[
\simpfverArrow{\id_x}{\id_x}{c}{a}{\id_x}{f}{\theta\psi}{\id_x}{\psi}\simpfverArrowcontinued{\theta}{x}{x}{x}{y},
\]
is of type $2$.
\end{ex}

\begin{rmk}
\label{simplicesDelta1}
Any $n$-simplex of $\Delta[1]$ is of the form $\chi_{k}\colon[n]\to[1]$ for a unique $k=-1,\dots,n$, with $\chi_k$ defined on objects by
\[\chi_k\colon i\mapsto\left\{
\begin{array}{rcl}
0&\mbox{ if }i\le k\\
1&\mbox{ if }i>k.
\end{array}\right.\]
\end{rmk}

\begin{rmk}
\label{NSigmaDoverDelta[1]} 
For any category $\cD$, the unique functor $!\colon\cD\to[0]$ induces a canonical functor $\Sigma\cD\to[1]$, which in turns defines a morphism of simplicial sets \[p_{\cD}\colon N\Sigma\cD\to\Delta[1],\]
given by
\[
p_{\cD}(\sigma)=\left\{
\begin{array}{lll}
\chi_{-1}&\text{ if $\sigma$ is maximally degenerate at $y$}\\
\chi_{k}&\text{ if $\sigma$ is of type $k$}\\
\chi_n&\text{ if $\sigma$ is maximally degenerate at $x$.}
\end{array}
\right.\]
Roughly speaking, the value of $p_{\cD}$ on an $n$-simplex $\sigma$ computes the type $k$.
\end{rmk}

We see by definition that that the Duskin nerve of the suspension $2$-category $\Sigma\cD$ has exactly two $0$-simplices, corresponding to the objects $x$ and $y$ of $\Sigma\cP$, and we want further to identify each (non-maximally-degenerate) $n$-simplex with a functor
$\sigma\colon [k]\times [l]^{\op} \to\cD$, which is in turn completely described by $(k+1)\times(l+1)$ objects of $\cD$ connected horizontally and vertically by morphisms of $\cD$
\[{
\scriptsize \begin{tikzcd}
 d_{0l} \arrow{d} \arrow{r} & d_{0 (l-1)} \arrow{d} \arrow{r} &  \cdots \arrow{r}   & d_{0 0} \arrow{d}\\
  d_{1 l} \arrow{d}\arrow{r} & d_{1 (l-1)} \arrow{d} \arrow{r} &  \cdots \arrow{r}  & d_{1 0 } \arrow{d}\\
\vdots \arrow{d}& \vdots \arrow{d}  & \ddots & \vdots \arrow{d} \\
d_{k l} \arrow{r}  & d_{k (l-1)} \arrow{r}  & \cdots   \arrow{r} & d_{k 0},
\end{tikzcd}
}\]
such that all the resulting squares are all commutative.

To this end, we first discuss how the collections of such morphisms assemble into a simplicial set.
We follow the convention that $[-1]=\varnothing$ is the empty category.

\begin{prop}\label{matrixsset}
Let $\cD$ be a category.
The collection
\[\matn n{\cD}:=\{\left(k,l,\sigma\colon [k]\times [l]^{\op} \to\cD\right)\ |\ k,l\ge-1,\ k+l=n-1\}\]
for $n\ge0$ defines a simplicial set $\mat{\cD}$ with respect to the following simplicial structure.
The faces and degeneracies of a $\cD$-valued matrix $\sigma\colon [k]\times [l]^{\op} \to\cD$ are given by
 \[
 d_i\sigma=
 \left\{\begin{array}{lr}
 \sigma (d^i\times\id_{[l]^{\op}}) \colon [k-1]\times [l]^{\op} \to \cD&\mbox{ for } 0\leq i\leq k,\\
 \sigma (\id_{[k]}\times d^{i-(k+1)})\colon [k]\times [l-1]^{\op} \to \cD&\mbox{ for } k+1\leq i\leq n;
 \end{array}\right.
 \]
 \[
 s_i\sigma=
 \left\{\begin{array}{lr}
 \sigma (s^i\times\id_{[l]^{\op}}) \colon [k+1]\times [l]^{\op} \to \cD &\mbox{ for } 0\leq i\leq k,\\
 \sigma (\id_{[k]}\times s^{i-(k+1)})\colon [k]\times [l+1]^{\op} \to \cD& \mbox{ for } k+1\leq i\leq n.
 \end{array}\right.
 \]
\end{prop}

\begin{proof}
The fact that $\mat{\cD}$ is a simplicial set can be verified by means of a straightforward computation involving the explicit given formulas for faces and degeneracies.
\end{proof}

\begin{rmk}
Given the isomorphism $[l]\cong[l]^{\op}$, to describe the set of $n$-simplices of $\mat\cD$ it is a priori not necessary to include the $\op$. However, including it simplifies the explicit formulas describing the simplicial structure of $\mat\cD$.
\end{rmk}

\begin{rmk}
The set $\matn{0}{\cD}$ has exactly two elements, given by \[(-1,0,!\colon\varnothing=[-1]\times[0]^{\op}\to\cD)\text{ and }(0,-1,!\colon\varnothing=[0]\times[-1]^{\op}\to\cD).\]
More generally, for any $n\ge0$ there is a unique element of $\matn{n}\cD$ with $l=-1$, namely $(n,-1,!)$, and a unique element with $k=-1$, namely $(-1,n,!)$, corresponding to the iterated degeneracies of the two $0$-simplices.
When instead $k,l\ge0$, the component $\sigma\colon[k]\times[l]^{\op}\to\cD$ determines uniquely the element $(k,l,\sigma)$ of $\matn n \cD$.
In conclusion, the set of $n$-simplices of $\mat\cD$ is given by 
 \[\matn n\cD=\{(-1,n,!)\}\amalg\coprod_{k,l\ge0,k+l=n-1}\{\sigma\colon[k]\times[l]^{\op}\to\cD\}\amalg\{(n,-1,!)\}.\] 
\end{rmk}

  \begin{rmk}
When $k,l\ge0$, the element $(k,l,\sigma)$ consists of an honest functor $[k]\times[l]^{\op}\to\cD$, which can be thought of as a \emph{matrix}\footnote{We warn the reader that the use of matrices from this paper is not directly related with the matrices used by Duskin in \cite{duskin}.} valued in $\cD$. By contrast, one may suggestively think of the $n$-simplices $(-1,n,!\colon[-1]\times[n]^{\op}\to\cD)$ and $(n,-1,!\colon[n]\times[-1]^{\op}\to\cD)$ as the \emph{empty row}, resp.~ the \emph{empty column}, of length $n$, and call all simplices of $\mat\cD$ \emph{matrices}. According to this interpretation, roughly speaking in the simplicial set $\mat{\cD}$:
 \begin{enumerate}[leftmargin=*]
    \item faces are given by removing precisely one row or one column of a matrix;
    \item degeneracies are given by doubling precisely one row or one column of a matrix;
    \item the non-degenerate simplices are the matrices where no two consecutive rows and no two consecutive columns coincide.
 \end{enumerate}
\end{rmk}

\begin{defn}
For $k,l\ge-1$, $n=k+1+l$ and $\cD$ a $1$-category, we call an element $(k,l,\sigma\colon[k]\times[l]^{\op}\to\cD)$ of $\matn n\cD$ a \emph{$\cD$-valued matrix}.
\end{defn}

\begin{convention}
For simplicity of exposition, we will refer to the matrix $(k,l,\sigma\colon[k]\times[l]^{\op}\to\cD)$ of $\matn n\cD$ as just $\sigma\colon[k]\times[l]^{\op}\to\cD$, assuming implicitly that $k$ and $l$ are part of the data. In particular, the matrix $[-1]\times[n]^{\op}\to\cD$ is different as an element of $\matn n \cD$ from the matrix $[n]\times[-1]^{\op}\to\cD$, if $n\ge0$ and $\cD$ is not empty.
\end{convention}

\begin{rmk}
\label{MatDoverDelta[1]}
There is an isomorphism of simplicial sets \[\Phi\colon\mat{[0]}\cong\Delta[1]\]
given by $(!\colon[k]\times[l]^{\op}\to[0])\mapsto \chi_k$.
More generally, for any category $\cD$ there is a map of simplicial sets
\[q_{\cD}\colon\mat\cD\to\mat{[0]}\cong\Delta[1]\]
given by
$q_{\cD}(\sigma\colon[k]\times[l]^{\op}\to\cD)= \chi_k$, making use of the notation from \cref{simplicesDelta1}.
Roughly speaking, the value of $q_{\cD}$ on an $n$-simplex $\sigma\colon[k]\times[l]^{\op}\to\cD$ encodes the number of rows $k$, and implicitly also also the number of columns as $l=n-k-1$.
 \end{rmk}

\begin{prop}
\label{matrixcoskeletal}
The simplicial set $\mat\cD$ is $3$-coskeletal.
\end{prop}

\begin{proof}
In order to show that $\mat\cD$ is $3$-coskeletal, we assume $n\ge4$ and show that any simplicial map $\tau\colon \partial\Delta[n] \to \mat{\cD}$ can be extended to an $n$-simplex $\widetilde{\tau}\colon\Delta[n]\to\mat\cD$ in $\mat\cD$, and that this can be done in a unique way.

The simplicial map ${\tau}\colon \partial\Delta[n] \to \mat{\cD}$ consists of a family of $(n-1)$-simplices $\tau_i$ of $\mat\cD$ for $0\leq i \leq n$, subject to the condition $d_i\tau_j=d_{j-1}\tau_i$ for $0\leq i < j \leq n$. Since $\Delta[1]$ is $1$-coskeletal, the simplicial map $q_{\cD}\circ\tau\colon \partial\Delta[n]\to\mat{\cD} \to \mat{[0]}\cong \Delta[1]$
extends uniquely to a simplicial map $\Delta[n]\to\mat{\cD} \to \mat{[0]}\cong \Delta[1]$, which by \cref{simplicesDelta1} must be of the form $\chi_k\colon\Delta[n]\to\Delta[1]$ for a unique $k$ with $-1\leq k\leq n$.
The fact that $q_{\cD}$ is compatible with the face structure maps implies that
\[q_{\cD}(\tau_i)=d_{i}(\chi_k)=\left\{
\begin{array}{cc}
    \chi_{k-1} & \text{if $0\leq i \leq k$} \\
   \chi_k & \text{if $k+1\leq i \leq n$,}
\end{array}\right.\]
and this relation determines the size of each matrix $\tau_i$, namely
\[\tau_i\text{ is of the form}\left\{
\begin{array}{cc}
   [k-1]\times [(n-2)-(k-1)]^{\op} \to \cD  & \text{if $0\leq i \leq k$} \\
   { [k]\times [(n-2)-k]^{\op} \to \cD }& \text{if $k+1\leq i \leq n$.}
\end{array}\right.
\]
We now construct a matrix of the form $\widetilde{\tau}\colon [k]\times [l]^{\op}\to \cD$ representing an $n$-simplex of $\mat\cD$ with boundary $\tau$. Note that if there exists an $n$-simplex $\widetilde\tau$ of $\mat\cD$ with boundary $\tau$, it must be of this form because $\partial(q_{\cD}\widetilde{\tau})=\partial(\chi_k)$.
 
 If $k=-1$ or $k=n$, we set $\widetilde\tau$ to be the empty row $!\colon[-1]\times[n]\to\cD$, or empty column $!\colon[n]\times[-1]\to\cD$ of length $n$, respectively. Note that these are the only simplices $\widetilde\tau$ with boundary $\tau$ and for which $k=-1$ or $k=n$. We can then assume $0\le k\le n-1$, and recall that $k+l=n-1\geq 3$.

We define $\widetilde\tau\colon[k]\times[l]^{\op}\to\cD$ on an object $(a,b)$ of $[k]\times [l]^{\op}$ 
as
\[
\widetilde\tau(a,b):=\tau_i(a',b'),
\]
where $0\leq i \leq k+1+l$ and $(a',b')$ are such that
\[(a,b)=\left\{
\begin{array}{ll}
  (d^ia',b')   &  \mbox{ if }0\leq i \leq k\text{ and }(a',b')\in[k-1]\times[l]^{\op}\\
   (a', d^{i-(k+1)}b')  & \mbox{ if }k+1\leq i \leq k+1+l \text{ and }(a',b')\in[k]\times[l-1]^{\op}.
\end{array}\right.\]
Note that if there exists a simplex $\widetilde\tau$ with boundary $\tau$, the value of objects must satisfy this condition.
To verify that the assignment is well-defined, one can use
\begin{itemize}[leftmargin=*]
    \item the fact that $k+l\ge3$ to see that such $i$ and $(a',b')$ always exist; roughly speaking, $i$ can be taken to be such that the $i$-th row or the $i-(k+1)$-th column of $[k]\times[l]^{\op}$ do not contain $(a,b)$;
    \item the fact that each coface $d^i\colon[k-1]\to[k]$ and each coface $d^{i-(k+1)}\colon[l-1]^{\op}\to[l]^{\op}$ is a monomorphism to see that $(a',b')$ is unique once $i$ is chosen; and
    \item the relations $d_i\tau_j=d_{j-1}\tau_i$ for $0\leq i < j \leq n$ and the simplicial identities for $\mat\cD$ to see that different choices of $i$ produce the same result.
\end{itemize}

Similarly, we define $\tau\colon[k]\times[l]^{\op}\to\cD$ on a morphism $(f,g)$ of $[k]\times [l]^{\op}$
as
\[
\tau(f,g):=\tau_i(f',g').
\]
where $0\leq i \leq k+1+l$ and $(f,'g')$ are such that
\[(f,g)=\left\{
\begin{array}{ll}
  (d^if',g')   &  \mbox{ if }0\leq i \leq k\text{ and }(f',g')\text{ is in }[k-1]\times[l]^{\op}\\
   (f', d^{i-(k+1)}g')  & \mbox{ if }k+1\leq i \leq k+1+l\text{ and }(f',g')\text{ is in }[k]\times[l-1]^{\op}.
\end{array}\right.\]
Again, one can verify that the assignment is well-defined. Note that if there exists a simplex $\widetilde\tau$ with boundary $\tau$, the value on morphisms must satisfy this condition.
Moreover, the assignment is by construction compatible with source, target and identities, and we now argue that it is compatible with composition.

To this end, we first prove that $\widetilde{\tau}$ is compatible with composition of morphisms $(a,b)\to (a',b')\to(a'',b'')$ in $[k]\times[l]^{\op}$
that involve at most four indices, namely such that $\{a,b,a',b',a'',b''\}$ has at most four elements. In this case, the pair of composable morphisms must be one of the following
{\small
\[
\begin{tikzcd}[column sep=0.4cm]
(a,b)\arrow[d]\\(a',b)\arrow[d]\\(a'',b)
\end{tikzcd}
,\quad
\begin{tikzcd}[column sep=0.4cm]
(a,b)\arrow[r] & (a,b') \ar[d] \\
& (a',b')
\end{tikzcd}
,\quad
\begin{tikzcd}[column sep=0.4cm]
(a,b)\arrow[d] & \\
 (a,b') \ar[r]& (a',b')
\end{tikzcd}
,\quad
\begin{tikzcd}[column sep=0.4cm]
(a,b) \arrow[r] & (a,b') \arrow[r] & (a,b'').
\end{tikzcd}
\]

}
In particular, there exists $0\leq i \leq k+1+l$ such that $(f,g)$, $(f',g')$ and their composite are in the image of 
$d^i\times\id $ if $0\leq i \leq k$ or in the image of
   $\id\times d^{i-(k+1)}$ if $k+1\leq i \leq k+1+l$. Roughly speaking, $i$ can be taken to be such that the $i$-th row or the $i-(k+1)$-th column of  $[k]\times[l]^{\op}$ do not contain the sources or targets of $(f,g)$, $(f',g')$ and their composite. In particular we can use $\tau_i$ to determine the value of $\widetilde{\tau}$ on these morphisms, and use that $\tau_i$ is a functor by assumption.
In particular, we conclude that $\widetilde\tau$ is compatible with this special kind of composition in $[k]\times[l]^{\op}$.
 
Now given an arbitrary pair of composable morphisms 
$(a,b)\to (a',b')\to(a'',b'')$ in $[k]\times[l]^{\op}$, we can prove that $\widetilde\tau$ is compatible with this composition using the previous consideration. Indeed, we have that:
{\scriptsize
\[
\begin{tikzcd}[column sep=0.4cm]
\widetilde{\tau}(a,b)\arrow[rd]& & \\
 & \widetilde{\tau}(a',b')\arrow[rd]&\\
 & & \widetilde{\tau}(a'',b'')
\end{tikzcd}
=
\begin{tikzcd}[column sep=0.4cm]
\widetilde{\tau}(a,b)\arrow[r]\ar[rd, dotted]& \widetilde{\tau}(a,b')\arrow[d]&\\
 & \widetilde{\tau}(a',b')\ar[rd, dotted] \arrow[r]& \widetilde{\tau}(a',b'')\ar[d]\\
 & & \widetilde{\tau}(a'',b'')
\end{tikzcd} 
\]
\[
=
\begin{tikzcd}[column sep=0.4cm]
\widetilde{\tau}(a,b)\arrow[r]& \widetilde{\tau}(a,b')\arrow[r] \ar[d, dotted]& \widetilde{\tau}(a,b'')\arrow[d]\\
 & \widetilde{\tau}(a',b')\arrow[r,dotted]& \widetilde{\tau}(a',b'')\ar[d]\\
 & & \widetilde{\tau}(a'',b'')
\end{tikzcd} 
=
\begin{tikzcd}[column sep=0.4cm]
\widetilde{\tau}(a,b)\arrow[rr, bend left ]\arrow[r,dotted]&\widetilde{\tau}(a,b')\arrow[r,dotted] & \widetilde{\tau}(a,b'') \arrow[dd, bend left=30]\arrow[d,dotted]\\
 & & \widetilde{\tau}(a',b'')\arrow[d, dotted]\\
 & & \widetilde{\tau}(a'',b'')
\end{tikzcd} 
\]
\[
=
\begin{tikzcd}[column sep=0.4cm]
\widetilde{\tau}(a,b)\arrow[rrdd, bend right ]\arrow[rr, bend left, dotted ]& &\widetilde{\tau}(a,b'') \arrow[dd, bend left=30, dotted]\\
 & & \\
 & & \widetilde{\tau}(a'',b'').
\end{tikzcd} 
\]
}
In particular, $\widetilde{\tau}$ is compatible with arbitrary composition of morphisms in $[k]\times[l]^{\op}$ and defines indeed a functor.
\end{proof}

The referee pointed out the following conceptual way of understanding the simplicial set $\mat\cD$, its coskeletality, the functoriality in $\cD$, and the canonical map to $\Delta[1]$.

\begin{aside}
The assignment $\cD\mapsto\mat\cD$ can be understood as the composite functor
\[
{\small{
\begin{tikzcd}[column sep=0.45cm]
\cat\arrow[r, "N"] & \set^{\Delta^{\op}}\arrow[r, "{d_*}"] & 
\set^{\Delta^{\op}\times\Delta^{\op}}\arrow[r,"{(\id\times\op)^*}"] &[0.6cm] \set^{\Delta^{\op}\times\Delta^{\op}} \arrow[r, "{D^{-1}}"] & \mathbf{C}(\Delta[0],\Delta[0])\arrow[r, "U"] &\set^{\Delta^{\op}},
\end{tikzcd}
}
}
\]
in which:
\begin{itemize}[leftmargin=*]
    \item $d_*$ denotes right Kan extension along the diagonal functor $d\colon \Delta \to \Delta \times \Delta$;
    \item $(\id\times\op)^*$ denotes the involution induced by $\op$ in the second variable; 
    \item the category $\mathbf{C}(\Delta[0], \Delta[0])$ is (referring to Joyal's notation from \cite[\textsection 7]{JoyalVolumeII}) the full subcategory of the slice category $\set^{\Delta^{\op}}/_{\Delta[1]}$ consisting of those simplicial sets over $\Delta[1]$ whose two fibers are isomorphic to $\Delta[0]$;
    \item $D^{-1}$ denotes the functor (which is in fact an instance of the equivalence from \cite[Prop.~7.4]{JoyalVolumeII}) that sends a bisimplicial set $X$ to the simplicial set $D^{-1}(X)$ over $\Delta[1]$ whose set of $n$-simplices is the disjoint union
\[D^{-1}X=\{*\}\amalg\coprod_{k,l\ge0,k+l=n-1}X_{k,l}\amalg\{*\}\]
and the map to $\Delta[1]$ is given in the evident way by the indexing of this sum. 
\item $U$ denotes the forgetful functor to simplicial sets.
\end{itemize}
In order to see the coskeletality of $\mat\cD$, one could use Reedy category theory and show that a bisimplicial set $X$ is $p$-coskeletal (in the sense of \cite[\textsection3]{RVreedy}) if and only if $D^{-1}X$ is $(p+1)$-coskeletal, and that for any category $\cD$ the bisimplicial set $d_*N\cD$ is $2$-coskeletal.
\end{aside}

As announced in \cref{ThmA}, we now identify the Duskin nerve of the suspension $2$-category $\Sigma\cD$ with the simplicial set of $\cD$-valued matrices. Recall that both $N\Sigma\cD$ and $\mat\cD$ are canonically endowed with maps $p_{\cD}$ and $q_{\cD}$ to $\Delta[1]$, as observed in \cref{NSigmaDoverDelta[1],MatDoverDelta[1]}.

 \begin{thm}
 \label{simplicesasmatrices}
Let $\cD$ be a $1$-category. There is an isomorphism of simplicial sets over $\Delta[1]$
$$\Phi_{\cD}\colon \mat{\cD}\cong N(\Sigma{\cD}).$$
  In particular,
  any $n$-simplex of $N\Sigma\cD$ can be described uniquely as a matrix $[k]\times [l]^{\op} \to \cD$ together with $k,l\ge-1$ and $k+l=n-1$; moreover, there is a bijective correspondence between the $n$-simplices of $N\Sigma\cD$ of type $k$, and the matrices $[k]\times[n-1-k]^{\op}\to\cD$.
  \end{thm}

\begin{proof}
We recall that the Duskin nerve of $\Sigma{\cD}$ is $3$-coskeletal, and we showed in \cref{matrixcoskeletal} that the set of $\cD$-valued matrices also assembles into a $3$-coskeletal simplicial set. Therefore, to construct the isomorphism $\Phi_{\cD}$ it is enough to identify the simplices of these two simplicial sets up to dimension $3$ compatibly with the simplicial structure.

We identify all simplices in dimension up to $3$ with $\cD$-valued matrices as follows.
\begin{enumerate}[leftmargin=*]
\item[(0)] Any of the two objects
$$x\quad\text{ and }\quad y$$
of $\Sigma{\cD}$ defines a $0$-simplex of the Duskin nerve of $\Sigma{\cD}$; we identify them respectively with the empty column $!\colon[0] \times [-1]^{\op}\to\cD$ and the empty row $!\colon[-1]\times [0]^{\op}\to\cD$ of length $0$. Similarly, for all $n=1,2,3$
any of the two objects $x$ and $y$ of $\Sigma{\cD}$ defines a unique degenerate $n$-simplex of the Duskin nerve; we identify them with the empty column $[n]\times [-1]^{\op}\to\cD$ and the empty row $[-1]\times [n]^{\op}\to\cD$ of length $n$,
respectively.
\item Any object $a$ in $\cD$ gives rise to a non-degenerate $1$-simplex in the Duskin nerve:
  \[
\begin{tikzcd}
    x \arrow[rr, "a"{below}]&& y
\end{tikzcd}
\]
    and all non-degenerate $1$-simplices of the Duskin nerve of $\Sigma{\cD}$ can uniquely be written in this form for some $a$ in $\cD$.
    We identify this $1$-simplex with the functor $[0]\times [0]^{\op}\to \cD$ with image
    \[a.\]
    \item Any morphism $\varphi\colon a\to b$ in $\cD$ gives rise to two $2$-simplices in the Duskin nerve of $\Sigma{\cD}$:
\[
\begin{tikzcd}[baseline=(current  bounding  box.center)]
 & x \arrow[rd, "b"] & \\
    x \arrow[ru, "{s_0x}"]
  \arrow[rr, "a"{below}, ""{name=D,inner sep=1pt}]
  && y
  \arrow[Rightarrow, from=D, to=1-2, shorten >= 0.1cm, shorten <= 0.1cm, "\varphi"] 
\end{tikzcd}
\mbox{ and }
\begin{tikzcd}[baseline=(current  bounding  box.center)]
 & y\arrow[rd, "{s_0y}"]  & \\
    x \arrow[ru, "b"]
  \arrow[rr, "a"{below}, ""{name=T,inner sep=1pt}]
  && y,
  \arrow[Rightarrow, from=T, to=1-2, shorten >= 0.1cm, shorten <= 0.1cm, "\varphi"]
\end{tikzcd}
\]
Moreover, all non-maximally-degenerate $2$-simplices of the Duskin nerve of $\Sigma{\cD}$ can uniquely be written in one of these two forms for some $\varphi\colon a\to b$ in $\cD$.
We identify these $2$-simplices with the functors $[1]\times [0]^{\op}\to\cD$ and $[0]\times [1]^{\op}\to\cD$
with image
\[ \begin{tikzcd}
a \arrow[d,"\varphi"] \\
 b 
\end{tikzcd}
\quad\quad
\text{and}
\quad\quad
\begin{tikzcd}
a\arrow[r,"\varphi"] & b.
\end{tikzcd}
\]
\item Any commutative square
\[\begin{tikzcd}
a \arrow[d,"\psi" swap]\arrow[r,"\varphi"]&b\arrow[d,"\gamma"] \\
 f \arrow[r,"\theta" swap]&c\\
\end{tikzcd}\]
in $\cD$ gives rise to three $3$-simplices in the Duskin nerve of $\Sigma\cD$:
\[
\simpfverArrow{\id_x}{\id_x}{c}{a}{\id_x}{f}{\theta\psi}{\id_x}{\psi}\simpfverArrowcontinued{\theta}{x}{x}{x}{y},
\]
\[
\simpfverArrow{\id_x}{c}{\id_y}{a}{b}{f}{\varphi}{\gamma}{\psi}\simpfverArrowcontinued{\theta}{x}{x}{y}{y},
\]
\[
\simpfverArrow{c}{\id_y}{\id_y}{a}{b}{\id_y}{\varphi}{\gamma}{\gamma\varphi}\simpfverArrowcontinued{\id_y}{x}{y}{y}{y}.
\]
Moreover, all non-maximally-degenerate $3$-simplices of the Duskin nerve can uniquely be written in one of these three forms for some commutative square in $\cD$ as above.
We identify these $3$-simplices with the functors $[2]\times [0]^{\op} \to \cD$, $[1]\times [1]^{\op} \to \cD$ and $[0]\times [2]^{\op} \to \cD$ displayed as
\[
\begin{tikzcd}
a \arrow[d,"\psi"] \\
 f \arrow[d,"\theta"] \\
 c
\end{tikzcd}
\quad\quad
\text{and}
\quad\quad
\begin{tikzcd}
a \arrow[d,"\psi" swap]\arrow[r,"\varphi"]&b\arrow[d,"\gamma"] \\
 f \arrow[r,"\theta" swap]&c\\
\end{tikzcd}
\quad\quad
\text{and}
\quad\quad
\begin{tikzcd}
a\arrow[r,"\varphi"] &b\arrow[r,"\gamma"]& c.
\end{tikzcd}
\]
    \end{enumerate}
    The given identification between simplices of the Duskin nerve in dimension up to $3$ and $\cD$-valued matrices can be checked to be compatible with the simplicial structure, using the explicit formulas from \cref{matrixcoskeletal}, and defines the desired isomorphism of simplicial sets $\Phi_{\cD}\colon\mat\cD\to N\Sigma\cD$.

Finally, we observe that there is a commutative diagram
\[\begin{tikzcd}
 \mat\cD\arrow[r,"\Phi_{\cD}"]\arrow[d,"p_{\cD}",swap]&N\Sigma\cD\arrow[d,"q_{\cD}"]\\
\mat[0]\arrow[r, "\Phi" swap]&\Delta[1]= N(\Sigma[0]), 
\end{tikzcd}\]
 which expresses the desired compatibility of $\Phi_{\cD}$ with the maps $p_{\cD}$ and $q_{\cD}$ to $\Delta[1]$.
The fact that the diagram commutes can be checked by means of the explicit formulas provided in \cref{NSigmaDoverDelta[1],MatDoverDelta[1]} on $n$-simplices for $n=0,1,2,3$, and deduced using a coskeletality argument for simplices in dimension higher than $3$.
\end{proof}

\section{The Duskin nerve of the free $2$-cell}
\label{section2}
As an instance of \cref{ThmA}, we obtain a full description of the non-degenerate simplices of the Duskin nerve of the free $2$-cell
\[
\begin{tikzcd}[row sep=3cm, column sep=2cm]
  x \arrow[r, bend left, "f_0", ""{name=U,inner sep=2pt,below}]
  \arrow[r, bend right, "f_1"{below}, ""{name=D,inner sep=2pt}]
  & y,
  \arrow[Rightarrow, from=U, to=D, "\alpha"]
\end{tikzcd}
\]
it being the suspension of the $1$-category $[1]$.

\begin{prop}
 \label{intrinsecdescription}
In dimension $n$, the Duskin nerve of the free $2$-cell has precisely two non-degenerate simplices $\sigma_{n}$ and $\sigma_n'$. More precisely, $\sigma_0:=y$ and $\sigma_0':=x$ are the two $0$-simplices of the Duskin nerve, $\sigma_1:=f_1\colon x\to y$ and $\sigma_1':=f_0\colon x\to y$ are the two $1$-simplices of the Duskin nerve, and for $n>1$ the $n$-simplices $\sigma_n$ and $\sigma_n'$ are described as follows.
\begin{itemize}[leftmargin=*]
    \item if $n=2m$, the non-degenerate $2m$-simplices $\sigma_{2m}$ and $\sigma_{2m}'$ are uniquely determined by the relations
     $$ d_i\sigma_{2m}=\left\{
     \begin{array}{lr}s_{m-1+i}\sigma_{2m-2}&\mbox{ for }0 \leq i \leq m-1 \\
    \sigma'_{2m-1}&\mbox{ for }i=m\\
    s_{i-m-1}\sigma_{2m-2}&\mbox{ for }m+1\leq i\leq 2m-1\\
    \sigma_{2m-1}&\mbox{ for }i=2m
    \end{array}\right.$$
    $$d_i\sigma_{2m}'=\left\{
    \begin{array}{lr}
          \sigma_{2m-1}&\mbox{ for }i=0\\
          s_{m-1+i}\sigma_{2m-2}'&\mbox{ for }1 \leq i \leq m-1 \\
    \sigma_{2m-1}'&\mbox{ for }i=m\\
    s_{i-m-1}\sigma_{2m-2}'&\mbox{ for }m+1\leq i\leq 2m;
    \end{array}
    \right.$$
    \item if $n=2m+1$, the non-degenerate $(2m+1)$-simplices $\sigma_{2m+1}$ and $\sigma_{2m+1}'$ are uniquely determined by the relations
    $$d_i\sigma_{2m+1}=\left\{\begin{array}{lr} s_{m+i}\sigma_{2m-1},&\mbox{ for }0 \leq i \leq m-1\\
    \sigma_{2m},&\mbox{ for }i=m\\
      \sigma_{2m}',&\mbox{ for }i=m+1\\
    s_{i-m-2}\sigma_{2m-1},&\mbox{ for }m+2\leq i\leq 2m+1
    \end{array}
    \right.
    $$
    $$
    d_i\sigma_{2m+1}'=\left\{\begin{array}{lr} \sigma_{2m},&\mbox{ for }i=0\\
    s_{m-1+i}\sigma_{2m-1}',&\mbox{ for }1 \leq i \leq m \\
    s_{i-m-1}\sigma_{2m-1}',&\mbox{ for }m+1\leq i\leq 2m\\
    \sigma_{2m}',&\mbox{ for }i=2m+1. 
    \end{array}
    \right.
    $$
\end{itemize}
\end{prop}

\begin{proof}
By \cref{simplicesasmatrices} we know that the Duskin nerve of the free $2$-cell has precisely two $0$-simplices and that the $n$-simplices that are not maximally degenerate can be enumerated by means of functors $[k]\times [l]^{\op}\to[1]$, with $k,l\ge0$ and $k+l=n-1$.
Moreover, an $n$-simplex is non-degenerate if and only if all rows are different and all columns are different. If $\sigma\colon[k]\times[l]^{\op}\to [1]$ is a non-degenerate simplex, it must be that $k+1\leq l+2$, given that each row is a functor $[l]^{\op}\to [1]$ and there are at most $l+2$ different ones, and similarly it must be that $l+1\leq k+2$, arguing with columns instead of rows.
Since we have $k+l=n-1$, we obtain that $n-2\leq 2k \leq n$.

According to this analysis,
in dimension $n$ the Duskin nerve of the free $2$-cells has precisely two non-degenerate simplices $\sigma_{n}$ and $\sigma_n'$.
\begin{itemize}[leftmargin=*] 
    \item If $n=2m$, the non-degenerate simplices $\sigma_{2m}$ and $\sigma'_{2m}$correspond to the
     functors
     \[\sigma_{2m}\colon[m-1]\times [m]^{\op}\to[1]\text{ and }\sigma_{2m}'\colon[m]\times [m-1]^{\op}\to[1]\]
given on objects by
\[
\sigma_{2m}(i,j)=\begin{cases}
0, \mbox{ if }i<j,\\
1, \mbox{ else.}
\end{cases}
\text{ and }\quad
\sigma_{2m}'(i,j)=\begin{cases}
0, \mbox{ if }i\leq j,\\
1, \mbox{ else.}
\end{cases}
\]
    \item If $n=2m+1$, the non-degenerate simplices $\sigma_{2m+1}$ and $\sigma'_{2m+1}$ correspond to the
     functors
     \[\sigma_{2m+1}\colon[m]\times [m]^{\op}\to[1]\text{ and }\sigma'_{2m+1}\colon[m]\times [m]^{\op}\to[1]\]
     given on objects by 
      \[
\sigma_{2m+1}(i,j)=\begin{cases}
0, \mbox{ if }i<j,\\
1, \mbox{ else.}
\end{cases}
\text{ and }\quad
\sigma_{2m+1}'(i,j)=\begin{cases}
0, \mbox{ if }i\leq j,\\
1, \mbox{ else.}
\end{cases}
\]
\end{itemize}
These matrices can be depicted as follows.
\[{
\small \begin{tikzcd}[row sep=small, column sep=small]
0 \arrow{d} \arrow{r} & 0 \arrow{d} \arrow{r} &  \cdots \arrow{r}&0\arrow{d}\arrow{r}& 0\arrow{d}\arrow{r}   & 1 \arrow{d}\\
  0 \arrow{d}\arrow{r} & 0 \arrow{d} \arrow{r} &  \cdots \arrow{r}&0\arrow{r} \arrow{d}& 1 \arrow{r} \arrow{d} & 1 \arrow{d}\\
    0 \arrow{d}\arrow{r} & 0 \arrow{d} \arrow{r} &  \cdots \arrow{r}& 1\arrow{r}\arrow{d}& 1\arrow{r}\arrow{d}  & 1 \arrow{d}\\
\vdots \arrow{d}& \vdots \arrow{d}  & \ddots & \vdots\arrow{d} &\vdots\arrow{d}&\vdots\arrow{d} \\
0\arrow{r} \arrow{d} &1 \arrow{d}\arrow{r}  & \cdots   \arrow{r} &1\arrow{r}\arrow{d}&1\arrow{r}\arrow{d}& 1\arrow{d}\\
1\arrow{r}  &1 \arrow{r}  & \cdots   \arrow{r} &1\arrow{r}&1\arrow{r}& 1
\end{tikzcd}
}
\mbox{ and }
{
\small \begin{tikzcd}[row sep=small, column sep=small]
0 \arrow{d} \arrow{r} & 0 \arrow{d} \arrow{r} &  \cdots \arrow{r}&0\arrow{d}\arrow{r}& 0\arrow{d}\arrow{r}   & 0 \arrow{d}\\
  0 \arrow{d}\arrow{r} & 0 \arrow{d} \arrow{r} &  \cdots \arrow{r}&0\arrow{r} \arrow{d}& 0 \arrow{r} \arrow{d} & 1 \arrow{d}\\
    0 \arrow{d}\arrow{r} & 0 \arrow{d} \arrow{r} &  \cdots \arrow{r}& 0\arrow{r}\arrow{d}& 1\arrow{r}\arrow{d}  & 1 \arrow{d}\\
\vdots \arrow{d}& \vdots \arrow{d}  & \ddots & \vdots\arrow{d} &\vdots\arrow{d}&\vdots\arrow{d} \\
0\arrow{r} \arrow{d} &0 \arrow{d}\arrow{r}  & \cdots   \arrow{r} &1\arrow{r}\arrow{d}&1\arrow{r}\arrow{d}& 1\arrow{d}\\
0\arrow{r}  &1 \arrow{r}  & \cdots   \arrow{r} &1\arrow{r}&1\arrow{r}& 1
\end{tikzcd}
}
\] 
In particular, in both matrices $\sigma$ and $\sigma'$ no two rows or columns are equal, and each row and column is increasing. 
 An induction argument using the simplicial identities shows that these simplices satisfy the desired relations.
The uniqueness of simplices satisfying those relations can be checked directly for simplices in dimension $1,2,3$ and follows from the $3$-coskeletality of $N\Sigma[1]$ (even $2$-coskeletality since $[1]$ is a poset) for simplices in dimension at least $4$.
\end{proof}

\begin{rmk}
As the Duskin nerve $N\Sigma[1]$ of the free $2$-cell is $2$-coskeletal, the simplices $\sigma_n$ and $\sigma_n'$ could be alternatively described in terms of their $2$-skeleta. For example, the $2$-dimensional faces of $\sigma_{2m}$ are identified by the following formula:
\[
\sigma_{2m}|_{\{ijk\}}=\left\{ \begin{array}{ll}
s_0^2\sigma_0'& \mbox{ if }  i<j<k \leq m-1,\\
s_0\sigma_1' & \mbox{ if } i<j<k-m,\\
\sigma_2' & \mbox{ if }  i<k-m\leq j\leq m-1,\\
s_0\sigma_1 &\mbox{ if } 0\leq k-m\leq i<j \leq m-1,\\
s_1\sigma_1' &\mbox{ if } i<j-m<k-m,\\
\sigma_2 & \mbox{ if }  0\leq j-m\leq i <k-m,\\
s_1\sigma_1 & \mbox{ if }0\leq j-m<k-m\leq i\leq m-1,\\
s_0^2\sigma_0& \mbox{ if } m-1<i<j<k.
\end{array}
\right.
\]
 \end{rmk}

\section{The Duskin nerve of $(r+1)$-point suspension $2$-categories}
\label{section3}

As announced informally in \cref{ThmB}, we now describe the Duskin nerve of $(r+1)$-point suspension $2$-categories $\Sigma[\cD_1,\dots,\cD_r]$, of which an important class of examples is given by all elements of Joyal's cell category $\Theta_2$. This description was inspired by the argument used in \cite[Prop.4.9]{rezkTheta}.

\begin{defn}
\label{suspension}
The \emph{$(r+1)$-point suspension} of given $1$-categories $\cD_1,\dots\cD_r$ is the $2$-category $\Sigma[\cD_1,\dots,\cD_r]$
with $r+1$ objects $x_0,\dots,x_r$ and hom categories given by
\[\Map_{\Sigma[ \cD_{1},\dots, \cD_{r}]}(x_{i},x_{j})\cong\left\{
\begin{array}{ccl}
{\cD_{i+1}} \times \ldots \times{\cD_{j}}&\text{ if }i<j\\
{[0]}&\text{ if }i=j\\
{[-1]}&\text{ if }i>j.
\end{array}
\right.\]
The only non-trivial compositions are given by the canonical isomorphisms
\[
\begin{tikzcd}[row sep=0.4cm]
 \Map_{\Sigma[ \cD_{1},\dots, \cD_{r}]}(x_{i},x_{j}) \times \Map_{\Sigma[ \cD_{1},\dots, \cD_{r}]}(x_{j},x_{k}) \arrow[r] \arrow[d, equals]&\Map_{\Sigma[ \cD_{1},\dots, \cD_{r}]}(x_{i},x_{k})\arrow[d,equals]\\
 \left({\cD_{i+1}} \times \ldots \times{\cD_{j}}\right) \times \left({\cD_{j+1}} \times \ldots \times{\cD_{k}}\right) \arrow[r, "\cong"] & {\cD_{i+1}} \times \ldots \times{\cD_{k}}.
\end{tikzcd}
\]
\end{defn}

\begin{rmk}
\label{rmk2}
Given $r\ge1$, we consider the
map of categories
\[s_r:=(\chi_0, \ldots, \chi_{r-1})\colon[r] \hookrightarrow [1]^r\]
given on objects by
\[s_r\colon i\mapsto(\underbrace{1,\dots,1}_{\text{$i$ times}},\underbrace{0,\dots,0}_{\text{$(r-i)$ times}}\!\!).\]
This map is injective on objects and fully faithful.
When taking nerves, the induced simplicial map
\[Ns_r\colon\Delta[r] \hookrightarrow \Delta[1]^r\]
is a monomorphism, and
an $n$-simplex $(\chi_{k_1},\dots,\chi_{k_r})\colon\Delta[n]\to\Delta[1]^r$ of $\Delta[1]^r$ is in the image of 
$Ns_r$ if and only if $k_1\le\dots\le k_r$. 
\end{rmk}

\begin{rmk}
\label{rmk1}
Given categories $\cD_1,\dots,\cD_r$, there are canonical maps of $2$-categories defined on $\Sigma[\cD_1,\dots,\cD_r]$.
\begin{enumerate}[leftmargin=*]
    \item For any $1\le i\le r$ there are canonical maps of $2$-categories
\[\Sigma[\cD_1,\dots,\cD_r]\to\Sigma[\cD_i],\]
which are induced by collapsing all $2$-categories $\Sigma\cD_j$ for $j<i$ to the point $x_{i-1}$ and all $2$-categories $\Sigma\cD_j$ for $j> i$ to the point $x_{i}$.
It in turns induces a map of simplicial sets
\[\mu_i\colon N(\Sigma[\cD_1,\dots,\cD_r])\to N(\Sigma[\cD_i]).\]
\item  There is a canonical map of $2$-categories
\[\Sigma[\cD_1,\dots,\cD_r]\to\Sigma[[0],\dots,[0]]\cong[r],\]
induced by collapsing to the terminal category each non-empty mapping category between two objects,
which in turns induces a map of simplicial sets
\[p_{\cD_1,\dots,\cD_r}\colon N\Sigma[\cD_1,\dots,\cD_r]\to \Delta[r],\]
 \end{enumerate}
\end{rmk}

\begin{prop}
\label{MultisuspensionPullback}
Given categories $\cD_1,\dots,\cD_r$, there is a pullback square of simplicial sets
\[
\begin{tikzcd}
N\Sigma[\cD_1,\dots,\cD_r] \arrow[rr,"{(\mu_1,\dots,\mu_r)}"]\arrow[d,swap,"p_{\cD_1,\dots,\cD_r}"] && N\Sigma[\cD_1] \times \ldots \times N\Sigma[\cD_r]\arrow[d,"p_{\cD_1}\times\dots\times p_{\cD_r}"]\\
 {\Delta[r]} \arrow[rr, "Ns_r",swap] && {\Delta[1]}\times\dots\times {\Delta[1]}.
\end{tikzcd}
\]
\end{prop}

\begin{proof}
We argue that there is a pullback square of $2$-categories
\[
\begin{tikzcd}
\Sigma[\cD_1,\dots,\cD_r]\arrow[r]\arrow[d] &\Sigma[\cD_1] \times \ldots \times \Sigma[\cD_r]\arrow[d]\\
 {[r]} \arrow[r, "s_r",swap] & {[1]}\times\dots\times {[1]},
\end{tikzcd}
\]
built using the canonical maps from \cref{rmk1,rmk2}.
From there we can then conclude, given that the Duskin nerve respects pullbacks and products, being a right adjoint.

The square of $2$-categories above commutes by direct inspection.
In order to prove that the square is a pullback of $2$-categories, we check that it is a pullback at the level of objects, and that it is a locally a pullback at the level of hom-categories of any pair of objects in $\Sigma[\cD_1,\dots,\cD_r]$.

At the level of objects, we ought to look at the commutative square of sets
\[
\begin{tikzcd}
\Ob(\Sigma[\cD_1,\dots,\cD_r])\arrow[r]\arrow[d] &\Ob(\Sigma[\cD_1] \times \ldots \times \Sigma[\cD_r])\arrow[d]\\
\Ob( {[r]})\arrow[r, "s_r",swap] &\Ob( {[1]}\times\dots\times {[1]}).
\end{tikzcd}
\]
This square is expressed as the following square, where both vertical maps are bijections,
\[
\begin{tikzcd}
\{x_0,\dots,x_r\}\arrow[r ]\arrow[d,"\cong"] &\{x,y\} \times \ldots \times \{x,y\}\arrow[d,"\cong"]\\
 \{0,\dots,r\} \arrow[r, "s_r",swap] & \{0,1\}\times\dots\times \{0,1\}.
\end{tikzcd}
\]
The square is therefore a pullback of sets.

At the level of hom-categories, given any two objects $x_i$ and $x_j$ of the $(r+1)$-point suspension $\Sigma[\cD_1,\dots,\cD_r]$, we ought to look at the commutative square of categories
\[
\begin{tikzcd}
\Map_{\Sigma[\cD_1,\dots,\cD_r]}(x_{i},x_{j})\arrow[r]\arrow[d]& \Map_{\Sigma[\cD_1] \times \ldots \times \Sigma[\cD_r]}(\vec x_{s_r(i)},\vec x_{s_r(j)})\arrow[d]\\
\Map_{[r]}(i,j) \arrow[r] & \Map_{{[1]}\times\dots\times {[1]}}(s_r(i),s_r(j)),
\end{tikzcd}
\]
where $\vec x_{s_r(i)}$ and $\vec x_{s_r(j)}$ denote the images of $x_i$ and $x_j$ in $\Sigma[\cD_1] \times \ldots \times \Sigma[\cD_r]$.
If $i>j$, this square is easily expressed as the following square, where the left vertical map is an isomorphism of empty categories $[-1]$,
\[
\begin{tikzcd}
{[-1]}\arrow[r]\arrow[d,"\cong"] &\Map_{\Sigma[\cD_1] \times \ldots \times \Sigma[\cD_r]}(\vec x_{s_r(i)},\vec x_{s_r(j)})\arrow[d]\\
 {[-1]} \arrow[r] & \Map_{{[1]}\times\dots\times {[1]}}(s_r(i),s_r(j)).
\end{tikzcd}
\]
If instead $i\le j$, this square is easily expressed as the following square, where both horizontal arrows are isomorphisms of categories,
\[
\begin{tikzcd}
\Map_{\Sigma[\cD_{1},\dots,\cD_{r}]}(x_i,x_{j})\arrow[r,"\cong"]\arrow[d] &{\cD_{i+1}} \times \ldots \times{\cD_{j}}\arrow[d]\\
{[0]} \arrow[r, "\cong"] & {[0]}\times\dots\times {[0]}.
\end{tikzcd}
\]
The square is therefore a pullback of categories in both cases.
\end{proof}

We can now prove \cref{ThmB}.

\begin{thm}
\label{nerver+1suspension}
An $n$-simplex of the Duskin nerve of the $(r+1)$-point suspension $N\Sigma[\cD_1,\dots,\cD_r]$ can be uniquely described as an $r$-uple of matrices
$[k_i]\times [l_i]^{\op} \to \cD_i$
for $i=1,\dots,r$,
together with $k_i,l_i\ge-1$, $k_i+l_i=n-1$ and subject to the condition that $k_i\le k_j$ for $0\le i\le j\le r$. The simplicial structure of $N\Sigma[\cD_1,\dots,\cD_r]$ coincides with the one induced by $\mat{\cD_1}\times\ldots \times \mat{\cD_r}$.
\end{thm}

\begin{proof}
Combining \cref{MultisuspensionPullback} with \cref{simplicesasmatrices}, we obtain that the simplicial set $N(\Sigma[\cD_1,\dots,\cD_r])$ can be described as the pullback of simplicial sets
\[
\begin{tikzcd}
N\Sigma[\cD_1,\dots,\cD_r] \arrow[rrr,"{(\Phi_{\cD_1}^{-1}\mu_1,\dots,\Phi_{\cD_r}^{-1}\mu_r)}"]\arrow[d,"p_{\cD_1,\dots,\cD_r}",swap] &&& \mat{\cD_1} \times \ldots \times \mat{\cD_r}\arrow[d,"q_{\cD_1}\times\dots\times q_{\cD_r}"]\\
 {\Delta[r]} \arrow[rrr, "Ns_r",swap] &&& {\Delta[1]}\times\dots\times {\Delta[1]}.
\end{tikzcd}
\]
It follows from \cref{rmk1,rmk2} that an $n$-simplex of the Duskin nerve of the $(r+1)$-point suspension $\Sigma[\cD_1,\dots,\cD_r]$ can be uniquely described as a $r$-uple of matrices $[k_i]\times [l_i]^{\op} \to \cD_i$ for $i=1,\dots,r$,
together with $k_i,l_i\ge-1$, $k_i+l_i=n-1$ and subject to the condition that $k_i\le k_j$ for $0\le i\le j\le r$.
\end{proof}

\appendix

 \section{An explicit computation} 
The proof of \cref{simplicesasmatrices} relies on the coskeletality of the simplicial sets $N\Sigma\cD$ and $\mat{\cD}$, and does not enlighten how the correspondence between $\cD$-valued matrices $[k]\times[n-1-k]^{\op}\to\cD$ and $n$-simplices in the Duskin nerve of $\Sigma\cD$ really works for $n\ge4$.
In this expository section we illustrate with an example how one can reconstruct a matrix from a simplex and vice versa.

The next corollary describes a correspondence between \emph{triangulations} labeled in the $2$-faces of a given simplex of the Duskin nerve of $\Sigma\cD$, and \emph{monotone paths} inside the corresponding $\cD$-valued matrix.

For triangulations, we make use of the formalism from \cite[Ex.\ 2.2.15]{DK}.
A \emph{triangulation} $\cT$ of a convex $(n+1)$-gon with cyclically numbered vertices only contains triangles with vertices being vertices of the original polygon. To any such triangulation $\cT$, we can associate a simplicial subset $\Delta[\cT] \subset \sk_2\Delta[n]\subset \Delta[n]$ by choosing the $2$-faces corresponding to the triangles in the triangulation.

\begin{defn}
Let $n\ge2$. Given an $n$-simplex $\sigma$ of $N\Sigma\cD$ of type $k$ for $0\le k\le n-1$, a \emph{$\sigma$-labeled triangulation} consists of a triangulation $\cT$ of an $(n+1)$-gon that does not have any triangle completely contained neither in $\{0,\ldots, k\}$ nor in $\{k+1,\ldots, n\}$, together with the composite
\[\Delta[\cT] \hookrightarrow \Delta[n] \xrightarrow{\sigma} N\Sigma\cD.\]
\end{defn}

In particular, the definition requires a compatibility between the triangulation $\cT$ and the simplex $\sigma$, namely that the $2$-simplices in the image of the composite $\Delta[\cT]\to N\Sigma\cD$ above are not degeneracies of a $0$-simplex.

\begin{ex}
Given $\sigma$ a $3$-simplex of $N\Sigma\cD$ of type $1$ given by
\[
\simpfverArrow{\id_x}{c}{\id_y}{a}{b}{f}{\varphi}{\gamma}{\psi}\simpfverArrowcontinued{\theta}{x}{x}{y}{y},
\]
an example of a $\sigma$-labeled triangulation is
\[
\begin{tikzcd}[column sep=1.3cm, row sep=0.4cm, baseline=(current  bounding  box.center), ampersand replacement=\&]
 y \& y \arrow[l, "\id_y" swap] \\
 {} \&  \\
 x \arrow[r, "\id_x" swap]\arrow[uu, "a", ""{name=a031,inner sep=2pt, swap} ] \arrow[uur, "b"{near start, xshift=0.15cm}, ""{name=a02,inner sep=2pt, swap}]\& x \arrow[uu, "c" swap]
 \arrow[Rightarrow, from=a031, to=1-2, shorten >= 0.3cm, "\varphi" ]
 \arrow[Rightarrow, from=a02, to=3-2, shorten >= 0.3cm, "\gamma" swap ]
 \end{tikzcd}
 \]
\end{ex}

Recall e.g.~
from \cite[Def.~65]{VerityComplicialI} that a \emph{shuffle} of $\Delta[k]\times\Delta[l]$ for $k,l\ge0$ is a non-degenerate $(k+l)$-simplex of $\Delta[k]\times\Delta[l]$. 
An easy and useful characterization of these
is that they are precisely the 
functors\footnote{We choose to have $[l]^{\op}$ rather than $[l]$ in the second factor. This convention is more convenient to the setup of the paper and not restrictive, modulo consistent adaptation of the relation between the indices $\alpha(i)$ and $\beta(i)$ in the next formula.}
\[S:=(\alpha,\beta)\colon[k+l]\to[k]\times[l]^{\op}\]
that satisfy the \emph{ordinate summation property}: for all $0\le i\le k+l$
\[\alpha(i) + l-\beta(i) = i .\]

\begin{defn}
Given a $\cD$-valued matrix $M\colon[k]\times[n-1-k]^{\op}\to\cD$ with $0\le k\le n-1$, a \emph{monotone path inside the matrix $M$} consists of a shuffle $S\colon[n-1]\to[k]\times[n-1-k]^{\op}$, 
together with
the restriction \[[n-1]\stackrel{S}{\longrightarrow}[k]\times[n-1-k]^{\op}\stackrel{M}{\longrightarrow}\cD\]
of $M$ along $S$.
\end{defn}

\begin{ex}
If $M\colon[1]\times[1]^{\op}\to\cD$ is a functor given by
\[\begin{tikzcd}
a \arrow[d,"\psi" swap]\arrow[r,"\varphi"]&b\arrow[d,"\gamma"] \\
 f \arrow[r,"\theta" swap]&c
\end{tikzcd}\]
a monotone path inside $M$ is for instance
\[\begin{tikzcd}
a\arrow[r,"\varphi"]&b\arrow[d,"\gamma"] \\
 &c.\\
\end{tikzcd}\]
\end{ex}

\begin{cor} 
Let $\cD$ be a $1$-category, $n\ge2$, $\sigma$ an $n$-simplex of $N\Sigma\cD$ of type $k$ for $0\le k\le n-1$, and  $M_{\sigma}\colon[k]\times[n-1-k]^{\op}\to\cD$ the corresponding $\cD$-valued matrix according to \cref{simplicesasmatrices}.
There is a bijective correspondence
between $\sigma$-labeled triangulations $\Delta[\cT]\to N\Sigma\cD$ and monotone paths $P\colon[n-1]\to\cD$ inside $M_{\sigma}\colon[k]\times[n-1-k]^{\op}\to\cD$.
\end{cor}

The corollary is a direct consequence of \cref{simplicesasmatrices} along with the following combinatorial fact\footnote{
The lemma appears to be a variant of the classical fact that the Catalan number $C_{n-1}$ can be expressed in two equivalent ways: as the number of triangulations of an $(n+1)$-gon, or as the number of monotone lattice paths along the edges of a grid with $(n-1)\times (n-1)$ square cells which do not pass above the diagonal. However, we are not aware of a direct comparison with the statement of our lemma.
}.

\begin{lem}
Let $n\ge2$ and let $0\leq k \le n-1$. Then
there is a bijective correspondence
between triangulations $\Delta[\cT]\to\Delta[n]$ of an $(n+1)$-gon which do not have a triangle completely contained neither in $\{0,\ldots, k\}$ nor in $\{k+1,\ldots, n\}$ and shuffles $S\colon [n-1]\to [k]\times[n-1-k]^{\op}$.
\end{lem}

\begin{proof}[Proof of the lemma]
The lemma is proven by induction on $n\ge2$. If $n=2$ there is a unique triangulation of the $(2+1)$-gon given by
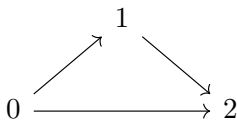
\[
\begin{tikzcd}[baseline=(current  bounding  box.center)]
 & 1 \arrow[rd] & \\
    0 \arrow[ru]
  \arrow[rr]
  && 2
\end{tikzcd}
\]
both when $k=0$ or $k=1$. On the other side, there is a unique shuffle given by $[1]\to [0]\times [1]^{\op}$ for $k=0$ and a unique shuffle $[1]\to [1]\times [0]^{\op}$ for $k=1$.

If $n>2$, we first show that for a given triangulation as above the edge $(0,n)$ is contained exactly in one triangle that is of the form $(0,n-1,n)$ or $(0,1,n)$ . To see this, assume otherwise that $(0,n)$ is contained in the triangle $(0,p,n)$ for some $1<p<n-1$. We only consider $p\leq k$, the other case being symmetric. Then the triangulation of the $(n+1)$-gon we started with induces a triangulation of the $(p+1)$-gon with vertices $0, 1, \ldots, p$, since $(0,p)$ is one of the edges in the triangulation. But we assumed that no triangles include only vertices in $0, 1, \ldots, k$, leading to a contradiction. 

It then follows 
that the given triangulation of the $(n+1)$-gon includes exactly one of the triangles $(0,n-1,n)$ and $(0,1,n)$, and is completely and uniquely described by such a triangle and the triangulation of the remaining $n$-gon.
By induction hypothesis, this corresponds to a shuffle of the form $[n-1-1]\to [k]\times [n-2-k]^{\op}$ in the first case and of the form $[n-1-1]\to [k-1]\times [n-1-k]^{\op}$ in the second case, together with an extra arrow that can be connected to this shuffle (horizontally in the first case and vertically in the second case).
By connecting these two pieces together we obtain a shuffle of the form $[n-1]\to [k]\times[n-1-k]^{\op}$, and all such shuffles arise in this way.
\end{proof}

We illustrate with an example how the proposition can be used to write down the matrix associated to a simplex and vice versa. The idea is that, given a triangulation labeled in a simplex, each simplex with a degenerate $2$-nd face contributes as a vertical step in the corresponding path and each $2$-simplex with a degenerate $0$-th face contributes as a horizontal step.

\begin{ex}
Let $\cP$ be a poset. Consider the $4$-simplex $\sigma$ of type $1$ of the Duskin nerve of $\Sigma\cP$
determined by the following $2$-skeleton
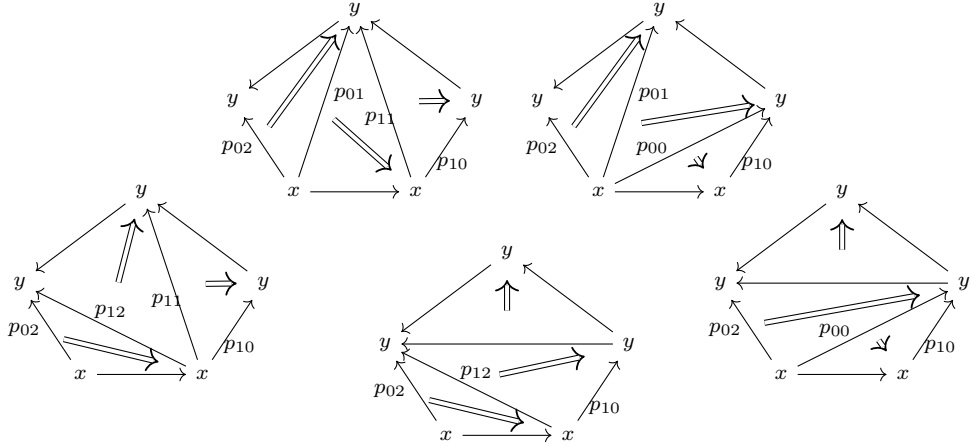
\[
 \begin{tikzpicture}[scale=.9, font=\scriptsize]
  \def\l{1.8cm}
\def\vertexa{x}
\def\vertexb{x}
\def\vertexc{y}
\def\vertexd{y}
\def\vertexe{y}
\def\edgeab{}
\def\edgebc{ p_{10}}
\def\edgecd{}
\def\edgede{}
\def\edgeae{p_{02}}
\def\edgebe{p_{12}}
\def\edgebd{p_{11}}
\def\edgead{p_{01}}
\def\edgece{}
\def\edgeac{p_{00}}
\def\triangleabe{}
\def\trianglebde{}
\def\trianglebcd{}
\def\triangleade{}
\def\triangleabd{}
\def\triangleacd{}
\def\triangleabc{}
\def\triangleace{}
\def\trianglecde{}
\def\trianglebce{}

    \begin{scope}  
  \draw[fill] (0,0) node (b0){$\vertexa$}; 
  \draw[fill] (\l,0) node (b1){$\vertexb$};
  \draw[fill] (1.5*\l,0.75*\l)  node (b2){$\vertexc$};
  \draw[fill] (0.5*\l,1.5*\l) node (b3){$\vertexd$};
  \draw[fill] (-0.5*\l,0.75*\l) node (b4){$\vertexe$};

  \draw[<-] (b1)--node[below]{$\edgeab$}(b0);
  \draw[->] (b1)--node[below,xshift=0.1cm]{$\edgebc$}(b2);
  \draw[->] (b2)--node[above]{$\edgecd$}(b3);
  \draw[->] (b3)--node[above]{$\edgede$}(b4);
    \draw[<-] (b4)--node[left](A2){}node[left](A3){$\edgeae$}(b0);
  \draw[->] (b1)--node[right](A4){}node[below,xshift=-0.06cm]{$\edgebd$}(b3);
  \draw[<-] (b3)--node[left](A1){}node[right, yshift=0.1cm]{$\edgead$}(b0);

    \draw[twoarrowlonger] (A1)--node[below]{$\triangleabd$}(b1);
\draw[twoarrowlonger] (A2)--node[below, yshift=-0.2cm]{$\triangleade$}(b3);
\draw[twoarrowlonger] (A4)--node[below]{$\trianglebcd$}(b2);
  \end{scope}
   
      \begin{scope}[xshift=2.5*\l] 
  \draw[fill] (0,0) node (b0){$\vertexa$}; 
  \draw[fill] (\l,0) node (b1){$\vertexb$};
  \draw[fill] (1.5*\l,0.75*\l)  node (b2){$\vertexc$};
  \draw[fill] (0.5*\l,1.5*\l) node (b3){$\vertexd$};
  \draw[fill] (-0.5*\l,0.75*\l) node (b4){$\vertexe$};

      \draw[<-] (b1)--node[below]{$\edgeab$}(b0);
  \draw[->] (b1)--node[below,xshift=0.1cm]{$\edgebc$}(b2);
  \draw[->] (b2)--node[below](B2){}node[above]{$\edgecd$}(b3);
  \draw[->] (b3)--node[above]{$\edgede$}(b4);
    \draw[<-] (b4)--node[left](B4){}node[left](A3){$\edgeae$}(b0);
  \draw[->] (b0)--node[left, xshift=0.2cm, yshift=-0.1cm](B1){}node[above, near start]{$\edgeac$}(b2);
  \draw[<-] (b3)--node[below](B3){}node[right, yshift=0.1cm]{$\edgead$}(b0);

    \draw[twoarrowlonger] ($(B3)+(0,-0.2cm)$)--node[above]{$\triangleacd$}(b2);
     \draw[twoarrowlonger] ($(B1)+(-0.1cm, 0.1cm)$)--node[above, xshift=0.2cm]{$\triangleabc$}(b1);
     \draw[twoarrowlonger] (B4)--node[below, yshift=-0.2cm]{$\triangleade$}(b3);
  \end{scope}
   
      \begin{scope}[xshift=4*\l, yshift=-1.5*\l]
  \draw[fill] (0,0) node (b0){$\vertexa$}; 
  \draw[fill] (\l,0) node (b1){$\vertexb$};
  \draw[fill] (1.5*\l,0.75*\l)  node (b2){$\vertexc$};
  \draw[fill] (0.5*\l,1.5*\l) node (b3){$\vertexd$};
  \draw[fill] (-0.5*\l,0.75*\l) node (b4){$\vertexe$};

          \draw[<-] (b1)--node[below]{$\edgeab$}(b0);
  \draw[->] (b1)--node[below,xshift=0.1cm]{$\edgebc$}(b2);
  \draw[->] (b2)--node[above]{$\edgecd$}(b3);
  \draw[->] (b3)--node[above]{$\edgede$}(b4);
    \draw[<-] (b4)--node[left](C2){}node[left](C3){$\edgeae$}(b0);
  \draw[->] (b0)--node[left, xshift=0.2cm, yshift=-0.1cm](C1){}node[above, near start]{$\edgeac$}(b2);
  \draw[->] (b2)--node[below]{$\edgece$}node[above](C4){}(b4);

\draw[twoarrowlonger] ($(C1)+(-0.1cm, 0.1cm)$)--node[above, xshift=0.2cm]{$\triangleabc$}(b1);
 \draw[twoarrowlonger] (C3)--node[below, near start]{$\triangleace$}($(b2)+(-0.5cm,-0.15cm)$);
  \draw[twoarrowlonger] (C4)--node[left]{$\trianglecde$}(b3);
  \end{scope}
   
      \begin{scope}[xshift=1.25*\l, yshift=-2*\l]
  \draw[fill] (0,0) node (b0){$\vertexa$}; 
  \draw[fill] (\l,0) node (b1){$\vertexb$};
  \draw[fill] (1.5*\l,0.75*\l)  node (b2){$\vertexc$};
  \draw[fill] (0.5*\l,1.5*\l) node (b3){$\vertexd$};
  \draw[fill] (-0.5*\l,0.75*\l) node (b4){$\vertexe$};

              \draw[<-] (b1)--node[below]{$\edgeab$}(b0);
  \draw[->] (b1)--node[below,xshift=0.1cm]{$\edgebc$}(b2);
  \draw[->] (b2)--node[above]{$\edgecd$}(b3);
  \draw[->] (b3)--node[above]{$\edgede$}(b4);
    \draw[<-] (b4)--node[left](D2){}node[left](D3){$\edgeae$}(b0);
  \draw[->] (b2)--node[below]{$\edgece$}node[above](D4){}(b4);
  \draw[->] (b1)--node[above]{$\edgebe$}node[above](D1){}(b4);

\draw[twoarrowlonger] (D3)--node[below, xshift=-0.2cm]{$\triangleabe$}($(b1)+(-0.5cm,0.15cm)$);
  \draw[twoarrowlonger] (D4)--node[left]{$\trianglecde$}(b3);
  \draw[twoarrowlonger] (D1)--node[below]{$\trianglebce$}($(b2)+(-0.5cm,-0.15cm)$);
  \end{scope}
   
      \begin{scope}[xshift=-1.75*\l, yshift=-1.5*\l]
  \draw[fill] (0,0) node (b0){$\vertexa$}; 
  \draw[fill] (\l,0) node (b1){$\vertexb$};
  \draw[fill] (1.5*\l,0.75*\l)  node (b2){$\vertexc$};
  \draw[fill] (0.5*\l,1.5*\l) node (b3){$\vertexd$};
  \draw[fill] (-0.5*\l,0.75*\l) node (b4){$\vertexe$};

              \draw[<-] (b1)--node[below]{$\edgeab$}(b0);
  \draw[->] (b1)--node[below,xshift=0.1cm]{$\edgebc$}(b2);
  \draw[->] (b2)--node[above]{$\edgecd$}(b3);
  \draw[->] (b3)--node[above]{$\edgede$}(b4);
    \draw[<-] (b4)--node[left](E2){}node[left](E3){$\edgeae$}(b0);

  \draw[->] (b1)--node[right](E4){}node[below,xshift=-0.06cm]{$\edgebd$}(b3);
  \draw[->] (b1)--node[above](E1){$\edgebe$}(b4);

\draw[twoarrowlonger] (E3)--node[below, xshift=-0.2cm]{$\triangleabe$}($(b1)+(-0.5cm,0.15cm)$);
\draw[twoarrowlonger] (E4)--node[below]{$\trianglebcd$}(b2);
\draw[twoarrowlonger] (E1)--node[left]{$\trianglebde$}(b3);
  \end{scope}
 \end{tikzpicture}\]
where $p_{ij}$ belongs to $\cP$ and let's determine the $\cP$-valued matrix $M_{\sigma}$ that corresponds to it according to \cref{simplicesasmatrices}.

Given that the edge $(1,2)$ of $\sigma$ is non-degenerate, the given $4$-simplex is of type $1$ and we can use \cref{simplicesasmatrices} to assert that the matrix $M_{\sigma}$ has to be of the form
\[M_{\sigma}\colon[1]\times[4-1-1]^{\op}=[1]\times[2]^{\op}\to\cP.\]
The $\sigma$-triangulation
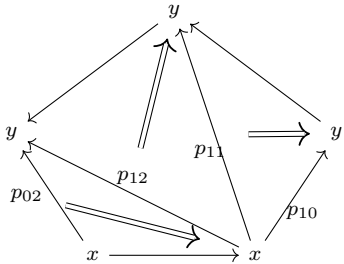
\[
\begin{tikzpicture}[scale=1.2, font=\scriptsize]
 \def\l{1.8cm}
\def\vertexa{x}
\def\vertexb{x}
\def\vertexc{y}
\def\vertexd{y}
\def\vertexe{y}
\def\edgeab{}
\def\edgebc{p_{10}}
\def\edgecd{}
\def\edgede{}
\def\edgeae{p_{02}}
\def\edgebe{p_{12}}
\def\edgebd{p_{11}}
\def\edgead{p_{01}}
\def\edgece{}
\def\edgeac{p_{00}}
\def\triangleabe{}
\def\trianglebde{}
\def\trianglebcd{}
\def\triangleade{}
\def\triangleabd{}
\def\triangleacd{}
\def\triangleabc{}
\def\triangleace{}
\def\trianglecde{}
\def\trianglebce{}
  \begin{scope}[xshift=-1.75*\l, yshift=-1.5*\l]
  \draw[fill] (0,0) node (b0){$\vertexa$}; 
  \draw[fill] (\l,0) node (b1){$\vertexb$};
  \draw[fill] (1.5*\l,0.75*\l)  node (b2){$\vertexc$};
  \draw[fill] (0.5*\l,1.5*\l) node (b3){$\vertexd$};
  \draw[fill] (-0.5*\l,0.75*\l) node (b4){$\vertexe$};

              \draw[<-] (b1)--node[below]{$\edgeab$}(b0);
  \draw[->] (b1)--node[below, xshift=0.1cm]{$\edgebc$}(b2);
  \draw[->] (b2)--node[above]{$\edgecd$}(b3);
  \draw[->] (b3)--node[above]{$\edgede$}(b4);
    \draw[<-] (b4)--node[left](E2){}node[left](E3){$\edgeae$}(b0);

  \draw[->] (b1)--node[right](E4){}node[below,xshift=-0.06cm]{$\edgebd$}(b3);
  \draw[->] (b1)--node[above](E1){$\edgebe$}(b4);

\draw[twoarrowlonger] (E3)--node[below, xshift=-0.2cm]{$\triangleabe$}($(b1)+(-0.5cm,0.15cm)$);
\draw[twoarrowlonger] (E4)--node[below]{$\trianglebcd$}(b2);
\draw[twoarrowlonger] (E1)--node[left]{$\trianglebde$}(b3);
  \end{scope}
 \end{tikzpicture}
 \]
corresponds to the monotone path in $M_{\sigma}$ that covers fully the left column and the bottom row and the $1$-st row, and is as follows:
\[{
\scriptsize \begin{tikzcd}
 p_{02} \arrow{d} & \\
  p_{1 2} \arrow{r} & p_{11}\arrow{r}  & p_{10}.
\end{tikzcd}
}\]
The $\sigma$-triangulation
\[
\begin{tikzpicture}[scale=1.2, font=\scriptsize]
 \def\l{1.8cm}
\def\vertexa{x}
\def\vertexb{x}
\def\vertexc{y}
\def\vertexd{y}
\def\vertexe{y}
\def\edgeab{}
\def\edgebc{p_{10}}
\def\edgecd{}
\def\edgede{}
\def\edgeae{p_{02}}
\def\edgebe{p_{12}}
\def\edgebd{p_{11}}
\def\edgead{p_{01}}
\def\edgece{}
\def\edgeac{p_{00}}
\def\triangleabe{}
\def\trianglebde{}
\def\trianglebcd{}
\def\triangleade{}
\def\triangleabd{}
\def\triangleacd{}
\def\triangleabc{}
\def\triangleace{}
\def\trianglecde{}
\def\trianglebce{}
  \begin{scope}  
  \draw[fill] (0,0) node (b0){$\vertexa$}; 
  \draw[fill] (\l,0) node (b1){$\vertexb$};
  \draw[fill] (1.5*\l,0.75*\l)  node (b2){$\vertexc$};
  \draw[fill] (0.5*\l,1.5*\l) node (b3){$\vertexd$};
  \draw[fill] (-0.5*\l,0.75*\l) node (b4){$\vertexe$};

  \draw[<-] (b1)--node[below]{$\edgeab$}(b0);
  \draw[->] (b1)--node[below, xshift=0.1cm]{$\edgebc$}(b2);
  \draw[->] (b2)--node[above]{$\edgecd$}(b3);
  \draw[->] (b3)--node[above]{$\edgede$}(b4);
    \draw[<-] (b4)--node[left](A2){}node[left](A3){$\edgeae$}(b0);
  \draw[->] (b1)--node[right](A4){}node[below, xshift=-0.06cm]{$\edgebd$}(b3);
  \draw[<-] (b3)--node[left](A1){}node[right, yshift=0.1cm]{$\edgead$}(b0);

    \draw[twoarrowlonger] (A1)--node[below]{$\triangleabd$}(b1);
\draw[twoarrowlonger] (A2)--node[below, yshift=-0.2cm]{$\triangleade$}(b3);
\draw[twoarrowlonger] (A4)--node[below]{$\trianglebcd$}(b2);
  \end{scope}
 \end{tikzpicture}
 \]
corresponds to the monotone path in $M_{\sigma}$ that goes through the $1$-st column, and is as follows:
\[{
\scriptsize \begin{tikzcd}
 p_{02} \arrow{r}&p_{01}\arrow{d}& \\
  & p_{11}\arrow{r}  & p_{10}.
\end{tikzcd}
}\]
The $\sigma$-triangulation
\[
\begin{tikzpicture}[scale=1.2, font=\scriptsize]
 \def\l{1.8cm}
\def\vertexa{x}
\def\vertexb{x}
\def\vertexc{y}
\def\vertexd{y}
\def\vertexe{y}
\def\edgeab{}
\def\edgebc{p_{10}}
\def\edgecd{}
\def\edgede{}
\def\edgeae{p_{02}}
\def\edgebe{p_{12}}
\def\edgebd{p_{11}}
\def\edgead{p_{01}}
\def\edgece{}
\def\edgeac{p_{00}}
\def\triangleabe{}
\def\trianglebde{}
\def\trianglebcd{}
\def\triangleade{}
\def\triangleabd{}
\def\triangleacd{}
\def\triangleabc{}
\def\triangleace{}
\def\trianglecde{}
\def\trianglebce{}
  \begin{scope}[xshift=2.5*\l] 
  \draw[fill] (0,0) node (b0){$\vertexa$}; 
  \draw[fill] (\l,0) node (b1){$\vertexb$};
  \draw[fill] (1.5*\l,0.75*\l)  node (b2){$\vertexc$};
  \draw[fill] (0.5*\l,1.5*\l) node (b3){$\vertexd$};
  \draw[fill] (-0.5*\l,0.75*\l) node (b4){$\vertexe$};

      \draw[<-] (b1)--node[below]{$\edgeab$}(b0);
  \draw[->] (b1)--node[below, xshift=0.1cm]{$\edgebc$}(b2);
  \draw[->] (b2)--node[below](B2){}node[above]{$\edgecd$}(b3);
  \draw[->] (b3)--node[above]{$\edgede$}(b4);
    \draw[<-] (b4)--node[left](B4){}node[left](A3){$\edgeae$}(b0);
  \draw[->] (b0)--node[left, xshift=0.2cm, yshift=-0.1cm](B1){}node[above, near start]{$\edgeac$}(b2);
  \draw[<-] (b3)--node[below](B3){}node[right, yshift=0.1cm]{$\edgead$}(b0);

    \draw[twoarrowlonger] ($(B3)+(0,-0.2cm)$)--node[above]{$\triangleacd$}(b2);
     \draw[twoarrowlonger] ($(B1)+(-0.1cm, 0.1cm)$)--node[above, xshift=0.2cm]{$\triangleabc$}(b1);
     \draw[twoarrowlonger] (B4)--node[below, yshift=-0.2cm]{$\triangleade$}(b3);
  \end{scope}
 \end{tikzpicture}
 \]
corresponds to the monotone path in $M_{\sigma}$ that covers fully the $0$-th row and the last column, and is as follows:
\[{
\scriptsize \begin{tikzcd}
 p_{02} \arrow[r] & p_{01} \arrow[r]& p_{00}\arrow[d]\\
  &  & p_{10}.
\end{tikzcd}
}\]

We conclude that $M_{\sigma}$ is the functor $[1]\times [2]^{\op}\to\cP$ given by
\[{
\scriptsize \begin{tikzcd}
 p_{02} \arrow{d} \arrow{r}&p_{01}\arrow{r}\arrow{d}&p_{00}\arrow{d} \\
  p_{1 2} \arrow{r} & p_{11}\arrow{r}  & p_{10}.
\end{tikzcd}
}\]
\end{ex}

\bibliographystyle{amsalpha}
\bibliography{ref}

\providecommand{\bysame}{\leavevmode\hbox to3em{\hrulefill}\thinspace}
\providecommand{\MR}{\relax\ifhmode\unskip\space\fi MR }
\providecommand{\MRhref}[2]{%
  \href{http://www.ams.org/mathscinet-getitem?mr=#1}{#2}
}
\providecommand{\href}[2]{#2}
\begin{thebibliography}{BGLS15}

\bibitem[BC03]{BullejosCegarra}
Manuel Bullejos and Antonio~M. Cegarra, \emph{On the geometry of 2-categories
  and their classifying spaces}, K-theory \textbf{29} (2003), no.~3, 211--229.

\bibitem[Ber07]{bergner}
Julia~E. Bergner, \emph{A model category structure on the category of
  simplicial categories}, Transactions of the American Mathematical Society
  \textbf{359} (2007), no.~5, 2043--2058.

\bibitem[BGLS15]{BuckleyGarnerLackStreet}
Mitchell Buckley, Richard Garner, Stephen Lack, and Ross Street, \emph{The
  {C}atalan simplicial set}, Math. Proc. Cambridge Philos. Soc. \textbf{158}
  (2015), no.~2, 211--222. \MR{3310241}

\bibitem[CCG10]{CCG}
Pilar Carrasco, Antonio~M. Cegarra, and Antonio~R. Garz{\'o}n, \emph{Nerves and
  classifying spaces for bicategories}, Algebr. Geom. Topol. \textbf{10}
  (2010), no.~1, 219--274. \MR{2602835 (2011d:18010)}

\bibitem[Ceg11]{cegarra}
Antonio~M. Cegarra, \emph{Homotopy fiber sequences induced by 2-functors},
  Journal of Pure and Applied Algebra \textbf{215} (2011), no.~4, 310--334.

\bibitem[DK19]{DK}
Tobias {Dyckerhoff} and Mikhail {Kapranov}, \emph{{Higher Segal spaces.}}, vol.
  2244, Cham: Springer, 2019 (English).

\bibitem[Dus02]{duskin}
John~W. Duskin, \emph{Simplicial matrices and the nerves of weak
  {$n$}-categories. {I}. {N}erves of bicategories}, Theory Appl. Categ.
  \textbf{9} (2001/02), 198--308, CT2000 Conference (Como). \MR{1897816
  (2003f:18005)}

\bibitem[Joy97]{JoyalDisks}
Andr{\'e} Joyal, \emph{{Disks, Duality and $\Theta$-categories}}, preprint
  available at \url{https://ncatlab.org/nlab/files/JoyalThetaCategories.pdf}
  (1997).

\bibitem[Joy07]{Joyal2007}
\bysame, \emph{Quasi-categories vs simplicial categories}, preprint available
  at
  \url{http://www.math.uchicago.edu/~may/IMA/Incoming/Joyal/QvsDJan9(2007).pdf},
  2007.

\bibitem[Joy08]{JoyalVolumeII}
\bysame, \emph{The theory of quasi-categories and its applications}, preprint
  available at
  \url{http://mat.uab.cat/~kock/crm/hocat/advanced-course/Quadern45-2.pdf},
  2008.

\bibitem[Lur09]{htt}
Jacob Lurie, \emph{Higher topos theory}, Annals of Mathematics Studies, vol.
  170, Princeton University Press, Princeton, NJ, 2009. \MR{2522659}

\bibitem[Nan19]{nanda}
Vidit Nanda, \emph{Discrete {M}orse theory and localization}, J. Pure Appl.
  Algebra \textbf{223} (2019), no.~2, 459--488. \MR{3850551}

\bibitem[OR20]{ORfundamentalpushouts}
Viktoriya {Ozornova} and Martina {Rovelli}, \emph{{Fundamental pushouts of
  $n$-complicial sets}},
  \href{https://arxiv.org/abs/2005.05844}{arXiv:2005.05844} (2020).

\bibitem[Rez10]{rezkTheta}
Charles Rezk, \emph{A {C}artesian presentation of weak {$n$}-categories}, Geom.
  Topol. \textbf{14} (2010), no.~1, 521--571. \MR{2578310}

\bibitem[RV14]{RVreedy}
Emily Riehl and Dominic Verity, \emph{The theory and practice of {R}eedy
  categories}, Theory and Applications of Categories \textbf{29} (2014), no.~9,
  256--301.

\bibitem[RV20]{RiehlVerityNcoh}
\bysame, \emph{Recognizing quasi-categorical limits and colimits in homotopy
  coherent nerves}, Applied Categorical Structures (2020), 1--48.

\bibitem[Ste07a]{SteinerOrientals}
Richard Steiner, \emph{Orientals}, Categories in algebra, geometry and
  mathematical physics, Contemp. Math., vol. 431, Amer. Math. Soc., Providence,
  RI, 2007, pp.~427--439. \MR{2342840}

\bibitem[Ste07b]{SteinerSimpleOmega}
\bysame, \emph{Simple omega-categories and chain complexes}, Homology Homotopy
  Appl. \textbf{9} (2007), no.~1, 451--465. \MR{2299807}

\bibitem[Str87]{StreetOrientedSimplexes}
Ross Street, \emph{The algebra of oriented simplexes}, J. Pure Appl. Algebra
  \textbf{49} (1987), no.~3, 283--335. \MR{920944}

\bibitem[Ver08a]{VerityComplicialAMS}
Dominic Verity, \emph{Complicial sets characterising the simplicial nerves of
  strict {$\omega$}-categories}, Mem. Amer. Math. Soc. \textbf{193} (2008),
  no.~905, xvi+184. \MR{2399898}

\bibitem[Ver08b]{VerityComplicialI}
\bysame, \emph{Weak complicial sets. {I}. {B}asic homotopy theory}, Adv. Math.
  \textbf{219} (2008), no.~4, 1081--1149. \MR{2450607}

\bibitem[Wol74]{wolff}
Harvey Wolff, \emph{{V}-cat and {V}-graph}, Journal of Pure and Applied Algebra
  \textbf{4} (1974), no.~2, 123--135.

\end{thebibliography}

\end{document}